\documentclass{amsart}
\usepackage{amssymb}
\usepackage{amsthm}
\usepackage{url}
\usepackage{mathtools}
\mathtoolsset{showonlyrefs=true}
\allowdisplaybreaks
\usepackage{Article}
\usepackage{Symbols}
\usepackage{hyperref}
\usepackage[
	backend=biber,
	style=alphabetic,
	maxalphanames=4,
	sorting=nty,
	arxiv=abs,
	doi=false,
	url=false,
	isbn=false,
	eprint=false
]{biblatex}

\makeatletter
\renewcommand{\subsection}{ \@startsection{subsection}{2}{\z@}  {.5\linespacing\@plus.7\linespacing}  {.5\linespacing}  {\normalfont\bfseries}}
\makeatother

\counterwithin{equation}{section}

\renewcommand{\thetheorem}{\thesection.\arabic{theorem}}

\makeatletter
\newcommand{\undeclare@theorem}[1]{%
	\expandafter\let\csname #1\endcsname\relax
	\expandafter\let\csname end#1\endcsname\relax
}
\@for\@tempa:={theorem,definition,lemma,proposition,corollary,conj,remark,note,eg,fact}\do{%
	\undeclare@theorem{\@tempa}%
	\undeclare@theorem{\@tempa*}%
}
\let\c@theorem\relax
\let\thetheorem\relax
\makeatother

\theoremstyle{plain}
\newtheorem{theorem}{Theorem}[section]
\newtheorem{lemma}[theorem]{Lemma}
\newtheorem{proposition}[theorem]{Proposition}
\newtheorem{corollary}[theorem]{Corollary}
\newtheorem{fact}[theorem]{Fact}

\newtheorem*{theorem*}{Theorem}
\newtheorem*{lemma*}{Lemma}
\newtheorem*{proposition*}{Proposition}
\newtheorem*{corollary*}{Corollary}
\newtheorem*{conj*}{Conjecture}

\makeatletter
\newtheorem*{maintheorem@statement}{Main Theorem}
\newenvironment{maintheorem}
	{\phantomsection\def\@currentlabel{Main Theorem}\maintheorem@statement}
	{\endmaintheorem@statement}
\makeatother

\theoremstyle{definition}

\newtheorem{remark}[theorem]{Remark}

\newtheorem*{definition*}{Definition}

\DeclareFieldFormat[article]{book}{\mkbibemph{#1}}
\DeclareFieldFormat*{volume}{\mkbibbold{#1}}
\renewbibmacro{in:}{}
\DeclareFieldFormat{pages}{#1}

\renewbibmacro*{journal+issuetitle}{%
	\usebibmacro{journal}%
	\setunit*{\addspace}%
	\iffieldundef{series}
		{}
		{\newunit\printfield{series}\setunit{\addspace}}%
	\usebibmacro{volume+number+eid}%
	\setunit{\addspace}%
	\usebibmacro{issue+date}%
	\setunit{\addcolon\space}%
	\usebibmacro{issue}%
	\newunit
}

\DeclareFieldFormat*{title}{\mkbibemph{#1}}
\DeclareFieldFormat*{journaltitle}{\textnormal{#1}}

\title[The Constant in Thomae-Type Formulas]{The Constant in Thomae-Type Formulas for Eight Points on the Complex Projective Line}
\author{{\sc Nakano} Ryunosuke}
\address{Graduate School of Science, Hokkaido University, Sapporo 060-0810, Japan}
\email{nakano.ryunosuke.i3@elms.hokudai.ac.jp}

\begin{document}

\begin{abstract}
	We consider the family of cyclic fourfold covers \(C(x)\) of the complex projective line branched at eight points \(x=(x_1,\ldots,x_8)\).
	It is known that the period map identifies the configuration space \(X(2,8)\) of the branch points with a Zariski open subset of a quotient of the five-dimensional complex ball, and that the inverse of the period map is expressed projectively by \(105\) automorphic forms \(f_J\), which are proportional to the signed branch-point polynomials \(\hat x_J\) with a common scalar factor.
	We normalize an element \(\omega\) of the one-dimensional eigenspace of the deck transformation group acting on the holomorphic differentials on \(C(x)\).
	For the period \(\eta\) of \(\omega\), we determine the common scalar factor: it is the product of a constant expressed in terms of the gamma function and the square of a quadratic form in \(\eta\).
\end{abstract}

\maketitle

\section{Introduction}
For a hyperelliptic curve \(y^2 = \prod_j (z - x_j)\), Thomae's classical formula relates powers of the nonzero theta constants at its period matrix to products of differences of the branch points \(x_j\) and a determinant of periods; see \cite[Proposition 3.6]{Fay73}.
A generalization to the \(Z_N\) curves \(s^N = \prod_j (z-\lambda_j)\), where the \(\lambda_j\) are distinct and the degree of the polynomial is divisible by \(N\), is proved in \cite{Nak97}.
The period maps of several families of cyclic covers of \(\Ps^1\) take values in complex balls in the sense of Deligne--Mostow \cite{DM86} and Terada \cite{Ter83}.
Thomae-type formulas are established for such families; see \cite[Theorem~2.1]{MS10} and \cite[Theorem~2.2]{Shi11}.
In particular, a Thomae-type formula is established in \cite{MN26} for the cyclic fourfold covers of \(\Ps^1\) attached to the Deligne--Mostow admissible sequence \((\tfrac14,\tfrac14,\tfrac14,\tfrac14,\tfrac14,\tfrac34)\), where the period map takes values in a three-dimensional complex ball.

We study the family of cyclic fourfold covers
\[
	C(x)\colon\quad w^4 = \prod_{j=1}^{7}(z - x_j)
\]
of \(\Ps^1\) branched at \(x_1,\ldots,x_7,x_8 = \infty\).
The configuration space \(X(2,8)\) of the branch points is defined as the quotient
\[
	X(2,8) = \PGL(2,\C) \backslash
	\{ (x_1,\ldots,x_8) \in (\Ps^1)^8 \mid
	x_j \neq x_k\ (j \neq k) \},
\]
where \(\PGL(2,\C)\) acts diagonally on \((\Ps^1)^8\).
We represent a point of \(X(2,8)\) by \(x = (x_1,\ldots,x_8)\).
Sending \((x_1,x_7,x_8)\) to \((0,1,\infty)\) by an element of \(\PGL(2,\C)\), we obtain a complete set
\[
	\{(0,x_2,\ldots,x_6,1,\infty) \in (\Ps^1)^8 \mid x_j\neq 0,1\ (2 \leq j \leq 6),\
	x_j\neq x_k\ (2\le j<k\le 6)\}
\]
of representatives for \(X(2,8)\).
We write \(\varpi \colon \widetilde{X}(2,8) \to X(2,8)\) for the universal cover of \(X(2,8)\).
By \cite[Section~4]{MY93} and \cite[Proposition~3.5]{MT04}, the period map \(\per\) identifies \(X(2,8)\) with a Zariski open subset of the ball quotient \(\Gamma_{M}\backslash\mathbb{B}_5\).
The period map has a holomorphic lift \(\widetilde{\per}\colon\widetilde{X}(2,8)\to\mathbb{B}_5\).
Here \(\mathbb{B}_5\) is the five-dimensional complex ball defined by the quotient
\[
	\mathcal{B}/\mathbb{C}^\ast= \{\xi \in \mathbb{C}^6 \mid \Transpose{\overline{\xi}} U \xi < 0\}/\mathbb{C}^\ast = \{\xi \in \mathbb{P}^5 \mid \Transpose{\overline{\xi}} U \xi < 0\},
\]
where
\[
	U =
	\begin{pmatrix}
		I_2 &    &    &     \\
		    &    & -1 &     \\
		    & -1 &    &     \\
		    &    &    & I_2
	\end{pmatrix},\qquad I_2 =
	\begin{pmatrix}
		1 & 0 \\
		0 & 1
	\end{pmatrix},
\]
and \(\Gamma_{M}\) is the conjugate \(T\Gamma(1+i)T^{-1}\) of the principal congruence subgroup \(\Gamma(1+i)\) of level \(1+i\) of the Picard modular group \(\Gamma\) over \(\Z[i]\) by \(T\); see Section~\ref{sec:preliminaries} for the definitions of \(\Gamma(1+i)\) and \(T \in \GL(6,\mathbb{Q}[i])\).
We write \(\varGamma\) for Euler's gamma function and \(\Gamma\) for the Picard modular group.
Note that the period \(\eta = \per(x)\) can be lifted to a value in \(\mathcal{B}\) by taking a path in \(X(2,8)\), and we write \(\eta(\tilde x)\) for the value at \(\tilde x \in \widetilde{X}(2,8)\).
Matsumoto and Terasoma \cite{MT04} constructed \(105\) automorphic forms \(T_J^{(2)}\) on \(\mathbb{B}_5\), indexed by the \((2,2,2,2)\)-partitions \(J = \langle j_1j_2;j_3j_4;j_5j_6;j_7j_8 \rangle\) of \(\{1,\ldots,8\}\), that is, by the partitions of \(\{1,\ldots,8\} = \{j_1,\ldots,j_8\}\) into four blocks \(\{j_1,j_2\}, \{j_3,j_4\}, \{j_5,j_6\}, \{j_7,j_8\}\) with \(j_{2k-1} < j_{2k}\).
They proved in \cite[Theorem 6.4]{MT04} that the projective point defined by these forms coincides, through the period map, with the projective point defined by the \(105\) polynomials
\[
	x_J = \prod_{k=1}^{4}\left( x_{j_{2k-1}} - x_{j_{2k}} \right),
\]
where \(x = (x_1,\ldots,x_8) \in X(2,8)\) and the factor containing \(\infty\) is omitted.
Matsumoto, Minowa, and Nishimura \cite{MMN07} expressed all \(105\) forms, denoted by \(f_J\), as explicit polynomials in theta constants.
In their tables, to each form \(f_J\) there corresponds the polynomial \(\hat x_J = \varepsilon_J x_J\) with \(\varepsilon_J \in \{\pm1\}\).
The results of \cite{MT04,MMN07} determine the vector \((f_J(\eta))_J\) up to a common scalar factor: this vector is proportional to \((\hat x_J(x))_J\).
In this paper, we determine this common scalar factor.
Our main result is the following.

\begin{maintheorem}
	\label{thm:main}
	Let \(\tilde x \in \widetilde{X}(2,8)\), and put \(x = \varpi(\tilde x)\) and \(\eta = \eta(\tilde x) \in \mathcal{B}\).
	Then, for every \((2,2,2,2)\)-partition \(J\), we have
	\begin{equation}
		\label{eq:thomae}
		f_J(\eta) = \kappa\,(\Transpose{\eta}U\eta)^2\, \hat x_J(x),
		\qquad
		\kappa = -\frac{1}{2^{12}\,\varGamma(3/4)^{16}}.
	\end{equation}
\end{maintheorem}
In terms of Jacobi's theta constant \(\vartheta_{00}(\tau)=\sum_{k\in\Z}\exp(\pi i k^2\tau)\), we have \(\kappa = -\vartheta_{00}(i)^{16}/(8\pi)^4\), since \(\vartheta_{00}(i) = \pi^{1/4}/\varGamma(3/4)\).
The constant \(\kappa\) permits comparison with other Thomae-type formulas and with boundary values; see Remark~\ref{rem:limit-curve} for the six-point case of \cite{MN26}.
We also decompose a theta series on \(\mathfrak{S}_4\) arising from the limit curve of Section~\ref{sec:degeneration} into theta constants on the upper half-plane \(\mathbb{H}=\{\tau\in \C\mid \Im(\tau)>0\}\); see Proposition~\ref{prop:splitting}.

We explain our method.
The proof has two parts.
First, we show that for each \(J\) the ratio of \(f_J(\eta)\) to \((\Transpose{\eta}U\eta)^2 \hat x_J(x)\) is constant on \(X(2,8)\), and that this constant does not depend on \(J\); see Proposition~\ref{prop:constancy}.
The invariance of this ratio under the monodromy follows from the automorphy of the forms \(f_J\), which is established in \cite{MMN07}.
The constancy follows from Hartogs' theorem, since every holomorphic function on \(\Ps^5\) is constant.
At a boundary hyperplane where two points coincide, Deligne--Mostow's extension results give, on a finite local cover, a holomorphic extension of the period \(\eta\).
Since the ratio is invariant under the monodromy, the ratio descends from this cover and extends across the hyperplane.
By the equality of the two projective points established in \cite{MT04,MMN07}, the constant does not depend on \(J\).
Second, we evaluate the constant at the degenerate configuration \((0,0,0,c,c,c,1,\infty)\) with \(c \in (0,1)\).
We choose this configuration so that the polynomial attached to \(J^\ast = \langle14;25;36;78\rangle\) does not vanish and the computation of the period reduces to an evaluation of the elliptic modular function.
There the normalization of the limit curve \(w^4 = z^3(z-c)^3(z-1)\) has genus \(3\).
The theta constants split into products of theta constants on \(\mathbb{H}\) at the moduli
\[
	\frac{\tau^\ast}{2},\qquad -\frac{1}{2\tau^\ast},\qquad
	\frac{i}{2},
\]
where \(\tau^\ast = -i(u+v)/(u-v)\), and \(u\) and \(v\) are explicit period integrals of the limit curve; see Section~\ref{sec:degeneration} for their definitions.
Using classical quadratic transformations of the Gauss hypergeometric series, we identify \(\tau^\ast\) with a value at \(m = (1-\sqrt{c})/2\) of a Schwarz map of the Gauss hypergeometric differential equation with parameters \(\tfrac12,\tfrac12;1\); see Lemma~\ref{lem:hypergeometric}.
By using duplication formulas for Jacobi's theta constants, we complete the evaluation of the constant \(\kappa\); see Sections~\ref{sec:degeneration} and~\ref{sec:evaluation}.

\section{Preliminaries}
\label{sec:preliminaries}
In this section, we choose a complete set of representatives for \(X(2,8)\), recall the signs \(\varepsilon_J\) attached to the \((2,2,2,2)\)-partitions, fix the branch of \(w\) as in \cite[Section 2]{MT04}, and introduce the Hermitian form on \(\C^6\).
We then fix the notation for the periods and for the theta functions, define the Picard modular group and the period map, and record the automorphy of the forms \(f_J\) of \cite{MMN07} that we use in Section~\ref{sec:constancy}.

\subsection{The Configuration Space}

The action of \(\PGL(2,\C)\) on \(\Ps^1\) is sharply three-transitive: for any ordered triple of distinct points, there is a unique element of \(\PGL(2,\C)\) that carries it to \((0, 1, \infty)\).
Hence a complete set of representatives for \(X(2,8)\) is given by
\begin{align*}
	 & x = (x_1,\ldots,x_8) = (0, x_2,x_3,x_4,x_5,x_6, 1, \infty), \\
	 & (x_2,\ldots,x_6) \in \{ (x_2,\ldots,x_6) \in \C^5 \mid
	x_j \neq 0,1\,(2 \leq j \leq 6),\
	x_j \neq x_k\,(2 \leq j < k \leq 6) \}.
\end{align*}
The coordinates \((x_2,\ldots,x_6)\) identify \(X(2,8)\) with a Zariski open subset of \(\C^5 \subset \Ps^5\).
For such a representative, we consider the cover \(C(x)\colon w^4 = \prod_{j=1}^{7}(z-x_j)\), whose deck transformation group is generated by
\[
	\rho \colon C(x) \ni (z, w) \longmapsto (z, iw) \in C(x).
\]
The form \(\frac{dz}{w}\) spans the eigenspace of \(\rho^{\ast}\) for the eigenvalue \(-i\) in the space of holomorphic differentials on \(C(x)\), and this eigenspace is one-dimensional; cf.\ \cite[Proposition~2.1]{MT04}.
The symmetric group \(S_8\) acts on \(x = (x_1,\ldots,x_8)\) by \((\sigma\cdot x)_l = x_{\sigma^{-1}(l)}\); cf.\ \cite[Section 2.1]{MMN07}.
For a \((2,2,2,2)\)-partition \(J\), we write \(\sigma^{-1}(J)\) for the partition obtained by applying \(\sigma^{-1}\) to every entry of \(J\), so that \(x_J(\sigma\cdot x) = \pm x_{\sigma^{-1}(J)}(x)\).
The tables of \cite[Section 5.1 and Theorems 1--4]{MMN07} attach a sign \(\varepsilon_J \in \{\pm 1\}\) to each partition.
The tables list the forms \(f_J\) together with the polynomials \(\pm x_J\), for instance \(-x_{\langle 17;25;68;34\rangle}\) in entry \#30 of \cite[Theorem 2]{MMN07}.
We put
\[
	\hat{x}_J = \varepsilon_J\, x_J.
\]

We set
\[
	X_{\R} = \{ (x_2,\ldots,x_6) \in \R^5 \mid 0 < x_2 < \cdots < x_6 < 1 \} \subset X(2,8),
\]
and we call \(X_{\R}\) the real chamber.
We first work on \(X_{\R}\).
We continue \(w\) across the branch points through \(\mathbb{H}\) as in \cite[Section~2]{MT04}, and on the interval \(\ell_j = [x_j, x_{j+1}]\) \((1\leq j \leq 6)\) we fix the branch of \(w\) for our normalization \(x_8=\infty\) by
\[
	w\big|_{\ell_j} = \exp\left( \frac{(7-j)\pi i}{4} \right)
	\left( \prod_{k \leq j} (z - x_k) \prod_{k > j} (x_k - z) \right)^{1/4},
\]
where the positive real fourth root of the positive product is taken.
For the branch convention with eight finite branch points, see \cite[Section~2]{MT02} and Remark~\ref{rem:branch-convention}.
We write \(\mathcal{I}_j\) for the integral \(\int_{\ell_j} \frac{dz}{w}\).
This branch takes the positive real fourth root on \(z > 1\) and is continued through \(\mathbb{H}\).
If \(z\) moves from the right to the left of a branch point \(x_k\) through \(\mathbb{H}\), then \(\arg(z-x_k)\) changes from \(0\) to \(\pi\).
On \(\ell_j\), the \(7-j\) factors with \(k>j\) are continued in this way, and they contribute the factor \(\exp\bigl( \frac{(7-j)\pi i}{4} \bigr)\).

\subsection{The Hermitian Form}

We define, for \(\xi, \eta \in \C^6\),
\[
	h_{U}(\xi, \eta) = \Transpose{\overline{\xi}}\, U \eta,
	\qquad
	q_{U}(\eta) = \Transpose{\eta}\, U \eta,
	\qquad
	U =
	\begin{pmatrix}
		1 &   &    &    &   &   \\
		  & 1 &    &    &   &   \\
		  &   &    & -1 &   &   \\
		  &   & -1 &    &   &   \\
		  &   &    &    & 1 &   \\
		  &   &    &    &   & 1
	\end{pmatrix}.
\]
Since \(U\) is real symmetric of signature \((5,1)\), \(h_{U}\) is a Hermitian form of signature \((5,1)\).
The function \(q_{U}\) is the quadratic form associated with the symmetric bilinear form \(\Transpose{\xi}\, U \eta\).
The function \(q_{U}\) appears below in the embedding of \(\mathbb{B}_5\) into the Siegel upper half-space, and in \eqref{eq:thomae} in the form \(\Transpose{\eta}U\eta\).
Since we divide by \(q_{U}\) below, we record the following lemma.

\begin{lemma}
	\label{lem:nonvanishing}
	If \(\eta \in \C^6\) satisfies \(h_{U}(\eta, \eta) < 0\), then \(q_{U}(\eta) = \Transpose{\eta}U\eta \neq 0\).
\end{lemma}
\begin{proof}
	We write \(\eta = \alpha + i\beta\) with real \(\alpha,\beta\).
	Since \(U\) is real symmetric, \(h_{U}(\eta, \eta) = \Transpose{\alpha}U\alpha + \Transpose{\beta}U\beta < 0\).
	Suppose that \(q_{U}(\eta)=0\).
	Comparing the real and imaginary parts of \(q_{U}(\eta)\), we obtain \(\Transpose{\alpha}U\alpha = \Transpose{\beta}U\beta < 0\) and \(\Transpose{\alpha}U\beta = 0\).
	These two conditions force \(\alpha\) and \(\beta\) to be linearly independent, since a linear relation between them contradicts \(\Transpose{\alpha}U\beta = 0\) together with \(\Transpose{\alpha}U\alpha < 0\).
	Thus \(\alpha,\beta\) span a two-dimensional subspace on which \(U\) is negative definite; this contradicts the signature \((5,1)\) of \(U\).
	Hence we have \(q_{U}(\eta) \neq 0\).
\end{proof}

\subsection{Periods and Monodromy}

We recall from \cite{MY93,MT04,MMN07} the cycles, the period, and the monodromy of the family.

\begin{fact}
	\label{fact:MT04-periods}
	\begin{enumerate}
		\item (\cite[Section 2.1, Proposition 2.2]{MT04})
		      Let \(x \in X_{\R}\).
		      We regard each interval \(\ell_j\) \((1 \leq j \leq 6)\) as a path on \(C(x)\) via the branch fixed above.
		      The chains \(A_j = (1-\rho^2)\ell_j\) and \(B_j = \rho A_j\) form a basis of the eigenspace \(H_1(C(x),\Z)^-\) for the eigenvalue \(-1\) of \(\rho^2\) acting on \(H_1(C(x),\Z)\).
		\item (\cite[Proposition 2.6]{MT04})
		      We set
		      \[
			      \begin{gathered}
				      \mathfrak{a}_1 = A_1,\qquad
				      \mathfrak{a}_2 = A_1 + A_2 + B_2,\qquad
				      \mathfrak{a}_3 = \mathfrak{a}_2 + B_3,\\
				      \mathfrak{a}_4 = \mathfrak{a}_3 - A_4 + B_4,\qquad
				      \mathfrak{a}_5 = \mathfrak{a}_3 + A_5,\qquad
				      \mathfrak{a}_6 = \mathfrak{a}_5 + A_6 + B_6,
			      \end{gathered}
		      \]
		      and we write \(\mathfrak{b}_1,\ldots,\mathfrak{b}_6\) for the cycles given in \cite[Proposition 2.6]{MT04}.
		      Following \cite[Section 2.2]{MT04}, we write \(\langle\,{,}\,\rangle\) for one half of the intersection form on \(H_1(C(x),\Z)^-\).
		      The set \(\Sigma_1 = \{\mathfrak{a}_1,\ldots,\mathfrak{a}_6,\mathfrak{b}_1,\ldots,\mathfrak{b}_6\}\) is a symplectic basis of a sub-Hodge structure \(L_1 \subset H_1(C(x),\Z)^-\) that is principally polarized by \(\langle\,{,}\,\rangle\), and \(\langle \mathfrak{a}_j, \mathfrak{b}_j\rangle = -1\) for \(1 \leq j \leq 6\).
		      Moreover, \(\rho\) acts on \(\Sigma_1\) by
		      \[
			      \Transpose{(\mathfrak{a}_1,\ldots,\mathfrak{a}_6,\mathfrak{b}_1,\ldots,\mathfrak{b}_6)}
			      \longmapsto
			      \begin{pmatrix} 0 & -U \\ U & 0 \end{pmatrix}
			      \Transpose{(\mathfrak{a}_1,\ldots,\mathfrak{a}_6,\mathfrak{b}_1,\ldots,\mathfrak{b}_6)}.
		      \]
		\item (\cite[Sections 3.1 and 3.2]{MT04}, \cite[Section 2.3]{MMN07})
		      We extend \(\langle\,{,}\,\rangle\) to a bilinear form on \(H_1(C(x),\R)^-\), and we give this space the complex structure for which \(\rho\) acts as multiplication by \(-i\).
		      The form \(h(\delta_1,\delta_2) = \langle \delta_1, \rho \delta_2\rangle - \langle \delta_1, \delta_2\rangle\, i\) is a Hermitian form of signature \((5,1)\), and the period
		      \[
			      \eta = \Transpose{\bigl( \int_{\mathfrak{a}_1} \frac{dz}{w},\ldots,\int_{\mathfrak{a}_6} \frac{dz}{w} \bigr)}
		      \]
		      satisfies \(h_{U}(\eta, \eta) < 0\).
		\item (\cite[Lemma 3.6]{MT04}, \cite[Proposition 3.1]{MY93}, \cite[Section 2.3]{MMN07})
		      The chains of assertion (1) continue along every path in \(X(2,8)\), and the continuation depends only on the homotopy class of the path.
		      The half twist interchanging two branch points \(x_j\) and \(x_k\) acts on the continued cycles by a unitary reflection \(g_{jk}\) of order four, and the full twist acts by the reflection \(g_{jk}^2\) of order two.
	\end{enumerate}
\end{fact}

By Fact~\ref{fact:MT04-periods}(2), the matrix \(U\) occurs as a block in the matrix of \(\rho\) with respect to \(\Sigma_1\); the matrix \(U\) is not the matrix of the intersection form on \(H_1(C(x),\Z)\).
Moreover, Fact~\ref{fact:MT04-periods}(2) gives \(h(\mathfrak{a}_j, \mathfrak{a}_k) = \langle \mathfrak{a}_j, \rho\mathfrak{a}_k\rangle = U_{jk}\), so the form \(h\) is expressed by \(h_{U}\) in the coordinates with respect to the basis \(\mathfrak{a}_1,\ldots,\mathfrak{a}_6\) of the complex vector space \(H_1(C(x),\R)^-\).
The period \(\eta\) is the coordinate vector of the linear form \(\delta \mapsto \int_\delta \frac{dz}{w}\) with respect to the dual basis, and since \(U^{-1} = U\), the form induced by \(h\) on the dual space is again expressed by \(h_{U}\) in these coordinates.
We write \(\eta_1,\ldots,\eta_6\) for the components of the period \(\eta\) of Fact~\ref{fact:MT04-periods}(3), so that \(\eta = \Transpose{(\eta_1,\ldots,\eta_6)}\) with \(\eta_j = \int_{\mathfrak{a}_j}\frac{dz}{w}\).

Fact~\ref{fact:MT04-periods} defines \(\eta\) only on \(X_{\R}\), and we extend \(\eta\) to the whole of \(X(2,8)\) by analytic continuation.
We fix a base point \(\dot x \in X_{\R}\).
For any \(x \in X(2,8)\), we choose a path from \(\dot x\) to \(x\), and we define the cycles \(\mathfrak{a}_1,\ldots,\mathfrak{a}_6\) on \(C(x)\) by continuing the cycles on \(C(\dot x)\) along that path.
The integrals \(\eta_j = \int_{\mathfrak{a}_j}\frac{dz}{w}\) continue with these cycles and are holomorphic in \(x\), so \(\eta\) is a multivalued holomorphic function on \(X(2,8)\).
The continuation preserves the intersection form and commutes with \(\rho\), hence preserves \(h\).
The Hodge--Riemann bilinear relations then give the inequality of Fact~\ref{fact:MT04-periods}(3) at every point of \(X(2,8)\).
By Fact~\ref{fact:MT04-periods}(4), the period \(\eta\) depends only on the homotopy class of the path from \(\dot x\) to \(x\), and two branches of \(\eta\) differ by the monodromy of the family.
We take \(\dot x\) as the base point of the universal cover \(\varpi \colon \widetilde{X}(2,8) \to X(2,8)\), and for \(\tilde x \in \widetilde{X}(2,8)\) we write \(\eta(\tilde x)\) for the continuation of \(\eta\) along a path representing \(\tilde x\).
We thus obtain a single-valued holomorphic map
\[
	\widetilde{X}(2,8) \ni \tilde x \longmapsto \eta(\tilde x) \in \mathcal{B},
	\qquad
	\mathcal{B} = \{ \xi \in \C^6 \mid h_{U}(\xi, \xi) < 0 \}.
\]
Since \(X_{\R}\) is convex, it is simply connected.
Hence the base point determines a unique lift of \(X_{\R}\) to \(\widetilde{X}(2,8)\).
On this lift, the components of \(\eta(\tilde x)\) are the integrals \(\eta_j = \int_{\mathfrak{a}_j}\frac{dz}{w}\) of Fact~\ref{fact:MT04-periods}, with the cycles and the branch fixed above.
We use \(X_{\R}\) again in Section~\ref{sec:degeneration}, where we take the limit of \(\eta\) at the degenerate configuration along the closure of \(X_{\R}\).

\subsection{The Picard Modular Group and the Period Map}

Following \cite[Section 2.3]{MMN07}, we define the Hermitian matrix \(H_6\) of size \(6\) by
\[
	\begin{gathered}
		(H_6)_{jj} = 2,
		\qquad
		(H_6)_{j,j+1} = -1+i,
		\qquad
		(H_6)_{j+1,j} = -1-i,\\
		(H_6)_{jk} = 0 \quad (|j-k| \geq 2),
	\end{gathered}
\]
and we write \(T = N_1 - iN_2\) for the coordinate change given in the same section, where
\begin{equation}
	\label{eq:N1-N2}
	N_1 =
	\begin{pmatrix}
		1 & 0 & 0 & 0  & 0 & 0 \\
		1 & 1 & 0 & 0  & 0 & 0 \\
		1 & 1 & 0 & 0  & 0 & 0 \\
		1 & 1 & 0 & -1 & 0 & 0 \\
		1 & 1 & 0 & 0  & 1 & 0 \\
		1 & 1 & 0 & 0  & 1 & 1
	\end{pmatrix},
	\qquad
	N_2 =
	\begin{pmatrix}
		0 & 0 & 0 & 0 & 0 & 0 \\
		0 & 1 & 0 & 0 & 0 & 0 \\
		0 & 1 & 1 & 0 & 0 & 0 \\
		0 & 1 & 1 & 1 & 0 & 0 \\
		0 & 1 & 1 & 0 & 0 & 0 \\
		0 & 1 & 1 & 0 & 0 & 1
	\end{pmatrix}.
\end{equation}
The entries of \(N_1\) and \(N_2\) are the coefficients of the cycles of Fact~\ref{fact:MT04-periods}:
\[
	\mathfrak{a}_j = \sum_{k=1}^{6}\bigl( (N_1)_{jk}A_k + (N_2)_{jk}B_k \bigr)
	\qquad (1 \leq j \leq 6).
\]
Since \((\rho^2)^{\ast}\frac{dz}{w} = -\frac{dz}{w}\) and \(\rho^{\ast}\frac{dz}{w} = -i\,\frac{dz}{w}\), we have \(\int_{A_k}\frac{dz}{w} = 2\mathcal{I}_k\) and \(\int_{B_k}\frac{dz}{w} = -2i\mathcal{I}_k\).
Hence we have \(\eta = 2T\,\Transpose{(\mathcal{I}_1,\ldots,\mathcal{I}_6)}\) on the real chamber.
A direct computation gives \(\det T = 2+2i\) and
\[
	TH_6T^{\ast} = 2U,
	\qquad
	(T^{-1})^{\ast} H_6^{-1} T^{-1} = \tfrac{1}{2}\,U.
\]
Therefore the map \(\mathcal{B} \ni \eta \mapsto T^{-1}\eta \in \C^6\) identifies \(\mathcal{B}\) with \(\{\xi \in \C^6 \mid \Transpose{\overline{\xi}} H_6^{-1} \xi < 0\}\), and hence \(\mathbb{B}_5\) with the ball of \cite{MMN07}.
We define the Picard modular group and its principal congruence subgroup of level \(1+i\) by
\[
	\begin{gathered}
		\Gamma = \U(H_6,\Z[i]) = \{ g \in \GL(6,\Z[i]) \mid g H_6 g^{\ast} = H_6 \},\\
		\Gamma(1+i)
		= \ker\bigl( \Gamma \longrightarrow \GL(6,\Z[i]/(1+i)) \bigr),
	\end{gathered}
\]
and we transport them to the coordinates \(\eta\) by putting
\[
	\Gamma_{U} = T\,\Gamma\,T^{-1},
	\qquad
	\Gamma_{M} = T\,\Gamma(1+i)\,T^{-1}.
\]
The groups \(\Gamma_{U}\) and \(\Gamma_{M}\) are isomorphic to \(\Gamma\) and \(\Gamma(1+i)\), respectively, but distinct from them.
Since \(\Gamma(1+i)\) is a kernel, it is normal in \(\Gamma\), and \(\Gamma_{M}\) is normal in \(\Gamma_{U}\).
Since \(1+i = i(1-i)\), the ideals \((1+i)\) and \((1-i)\) of \(\Z[i]\) coincide, and the group \(\Gamma(1+i)\) is the group \(\Gamma(1-i)\) of \cite{MY93,MMN07}.
By \cite[Proposition 3.1 and Theorem 4.1]{MY93}, the group \(\Gamma(1+i)\) is generated by \(21\) of the reflections \(g_{jk}^2\).
The index pairs \((j,k)\) are those listed in \cite[Proposition 3.1]{MY93}.
A direct computation shows that the corresponding \(21\) matrices \(Tg_{jk}^2T^{-1}\) have entries in \(\Z[i]\) and preserve \(U\); hence \(\Gamma_{M}\) is a subgroup of \(\U(U,\Z[i])\).
With \(T = N_1 - iN_2\) as in \eqref{eq:N1-N2}, the congruence condition refers to the coordinates \(T^{-1}\eta\) and not to the coordinates \(\eta\).
The elements of \(\Gamma_{M}\) are not all congruent to \(I_6\) modulo \(1+i\); cf.\ \cite[Remark 1]{MT02}.
The deck transformations of \(\varpi\) act through the monodromy representation
\[
	\begin{gathered}
		\mu \colon \pi_1(X(2,8), \dot x) \longrightarrow \Gamma_{M} \subset \U(U,\Z[i]),\\
		\eta(\gamma \cdot \tilde x) = \mu(\gamma)\,\eta(\tilde x)
		\qquad (\gamma \in \pi_1(X(2,8), \dot x)),
	\end{gathered}
\]
and the image of \(\mu\) is the whole group \(\Gamma_{M}\) by \cite[Theorem 4.1]{MY93}.
By the equivariance above, the class of \(\eta(\tilde x)\) in \(\Gamma_{M}\backslash\mathbb{B}_5\) depends only on \(\varpi(\tilde x)\), and we obtain the period map
\[
	\per \colon X(2,8) \longrightarrow
	\Gamma_{M}\backslash \mathbb{B}_5,
	\qquad
	\mathbb{B}_5
	= \{\xi \in \Ps^5 \mid h_{U}(\xi, \xi) < 0\},
\]
from the configuration space \(X(2,8)\) to the ball quotient.
Composing the map \(\tilde x\mapsto\eta(\tilde x)\) with the projection \(\mathcal{B}\to\mathbb{B}_5\), we obtain a lift \(\widetilde{\per}\colon\widetilde{X}(2,8)\to\mathbb{B}_5\) of \(\per\).
The map \(\tilde x\mapsto\eta(\tilde x)\) is determined by the differential \(dz/w\) and the homology marking fixed above, and it fixes the scale of \(\eta(\tilde x)\) as well as the class \(\widetilde{\per}(\tilde x)\) of \(\eta(\tilde x)\) in \(\mathbb{B}_5\).
If \(F\) is a holomorphic function on \(\mathcal{B}\) invariant under \(\Gamma_{M}\), then the function \(\tilde x\mapsto F(\eta(\tilde x))\) on \(\widetilde{X}(2,8)\) is invariant under the deck transformations and descends to a holomorphic function on \(X(2,8)\).

\subsection{Theta Functions}

We write \(\mathfrak{S}_n\) for the Siegel upper half-space of degree \(n\), that is, the set of symmetric matrices \(\tau\) of size \(n\) over \(\C\) whose imaginary part is positive definite.
We set \(\e(t) = \exp(2\pi i t)\).
For \(n = 1\), we have \(\mathfrak{S}_1 = \mathbb{H}\).
For a subring \(\mathcal{R}\) of \(\C\), we set
\[
	\Sp(2n,\mathcal{R}) = \{ M \in \GL(2n,\mathcal{R}) \mid M J_{2n} \Transpose{M} = J_{2n} \},
	\qquad
	J_{2n} = \begin{pmatrix} O_n & -I_n \\ I_n & O_n \end{pmatrix},
\]
where \(I_n\) and \(O_n\) are the identity and zero matrices of size \(n\), respectively.
Thus \(\Sp(2n,\mathcal{R})\) consists of matrices of size \(2n\).
The group \(\Sp(2n,\R)\) acts on \(\mathfrak{S}_n\) by \(M\cdot\tau = (M_{11}\tau+M_{12})(M_{21}\tau+M_{22})^{-1}\), where we write \(M\) in \(n\times n\) blocks as \(M = \left(\begin{smallmatrix}M_{11}&M_{12}\\M_{21}&M_{22}\end{smallmatrix}\right)\).
We define Riemann's theta function with half characteristics \((a/2,b/2)\) in the variables \((\zeta,\tau)\in\C^n\times\mathfrak{S}_n\) by the series
\begin{equation}
	\label{eq:theta-series}
	\vt{a}{b}(\zeta,\tau)
	= \sum_{k\in\Z^n}
	\e\left(
	\frac12\left(k+\frac12 a\right)\tau\Transpose{\left(k+\frac12 a\right)}
	+ \left(k+\frac12 a\right)\Transpose{\left(\zeta+\frac12 b\right)}
	\right),
\end{equation}
where \(a,b\in\Z^n\).
The series converges absolutely and uniformly on any compact set in \(\C^n\times\mathfrak{S}_n\).
The same series converges for arbitrary real characteristics, and we use it with half-integral \(b\) in Proposition~\ref{prop:splitting}.
Note that the characteristics enter \eqref{eq:theta-series} as \(a/2\) and \(b/2\).
Thus, in \(\vt{110000}{001100}\), we have \(a = (1,1,0,0,0,0)\) and \(b = (0,0,1,1,0,0)\).
For \(n=1\) and real \(a\) and \(b\), we write \(\vartheta_{ab}(\tau) = \vt{a}{b}(0,\tau)\), and we separate the two subscripts by a comma when either is a compound expression.
For \(a,b\in\{0,1\}\), the functions \(\vartheta_{ab}\) are Jacobi's theta constants.

\subsection{The Automorphic Forms}

Following \cite[Section 2.3]{MMN07}, we embed the complex ball \(\mathbb{B}_5\) into \(\mathfrak{S}_6\) by
\[
	\imath(\eta) = i\left( U - \frac{2}{\Transpose{\eta}U\eta}\,\eta\,\Transpose{\eta} \right),
\]
and we write \(\vt{a}{b}(\eta) = \vt{a}{b}(0, \imath(\eta))\) for the theta constants with characteristics \(a, b \in \Z^6\).
Then we have the following fact.

\begin{fact}
	\label{fact:iota}
	\begin{enumerate}
		\item (\cite[Section 2.3]{MMN07})
		      The map \(\imath\) is well defined on \(\mathcal{B}\) with values in \(\mathfrak{S}_6\), and \(\imath(\eta)\) depends only on the class of \(\eta\) in \(\mathbb{B}_5\).
		\item (\cite[Section 2.3]{MMN07})
		      For \(g \in \Gamma_{U}\), we write \(g = A + iBU\) with real matrices \(A, B\).
		      The assignment
		      \[
			      \jmath(g)
			      = \begin{pmatrix} A & B \\ -U B U & U A U \end{pmatrix}
		      \]
		      defines a homomorphism into \(\Sp(12,\Q)\) satisfying \(\imath(g\eta) = \jmath(g) \cdot \imath(\eta)\).
		\item (\cite[Section 3.2]{MT04}, \cite[Proposition 6]{MT02}, \cite[Section 2.3]{MMN07})
		      For the period \(\eta\), the matrix \(\imath(\eta)\) is the normalized period matrix of the principally polarized abelian variety attached to \((L_1, \Sigma_1)\).
	\end{enumerate}
\end{fact}

By Fact~\ref{fact:iota}(3), the results of \cite{MT04,MMN07} hold at \(\imath(\eta)\), so that the values \(\vt{a}{b}(\eta)\) are the theta constants of that abelian variety.

\begin{remark}
	\label{rem:branch-convention}
	We follow the convention of \cite{MT04,MT02}: the branches of \(w\) fixed above are obtained by analytic continuation through \(\mathbb{H}\).
	In \cite[Section 2]{MT02}, where all eight branch points are taken to be finite, the lift of \([x_j,x_{j+1}]\) is specified by the condition \(\exp(\frac{\pi i}{4}j)\,w \in [0,\infty)\).
\end{remark}

The forms \(f_J\) of \cite[Section 5]{MMN07} that we use are attached to the five double coset representatives
\begin{gather*}
	J_1 = \langle 12;34;56;78 \rangle,\quad
	J_2 = \langle 13;24;56;78 \rangle,\quad
	J_3 = \langle 13;25;46;78 \rangle,\\
	J_4 = \langle 13;24;57;68 \rangle,\quad
	J_5 = \langle 13;25;47;68 \rangle
\end{gather*}
of \(\operatorname{Stab}(J_1)\backslash S_8 / \operatorname{Stab}(J_1)\), where \(\operatorname{Stab}(J_1)\) is the stabilizer of \(J_1\) in \(S_8\).
These five forms are given by
\begin{align}
	f_{J_1} & = \vt{000000}{000000}\vt{001111}{001111}\vt{110011}{110011}\vt{111100}{111100},
	\\
	f_{J_2} & = 4\,\vt{110000}{000000}\vt{111111}{001111}\vt{000011}{110011}\vt{001100}{111100},
	\\
	f_{J_3} & = \left({\vt{111100}{000000}}^2 - {\vt{110000}{001100}}^2\right)
	\left({\vt{000011}{111111}}^2 - {\vt{001111}{110011}}^2\right),
	\label{eq:fJ-list}
	\\
	f_{J_4} & = \left({\vt{110011}{000000}}^2 - {\vt{110000}{000011}}^2\right)
	\left({\vt{001100}{111111}}^2 - {\vt{001111}{111100}}^2\right),
	\\
	f_{J_5} & = \frac{1}{4}\prod_{\epsilon_1,\epsilon_2 \in \{\pm 1\}}
	\Bigl( \vt{111111}{000000} + \epsilon_1\vt{110000}{001111}
	\\*
	& \hspace{7.2em}
	+ \epsilon_2\vt{001100}{110011} + \epsilon_1\epsilon_2\vt{000011}{111100} \Bigr)
\end{align}
(see \cite[Fact 3 and Propositions 4--7]{MMN07}).
For \(k = 1,\ldots,5\), the sign \(\varepsilon_{J_k}\) is \(+1\); hence \(\hat x_{J_k} = x_{J_k}\).
In the normalization of \cite[Section 5]{MMN07}, the forms \(f_J\) represent, as explicit polynomials in theta constants, the forms \(T_J^{(2)}\) constructed in \cite{MT04}.
By \cite[Fact 5]{MMN07}, which rests on \cite[Theorem 6.4]{MT04}, the collection \((f_J)_J\) is equivariant under \(S_8\), and for every \(\tilde x \in \widetilde{X}(2,8)\) the projective points \([\ldots : f_J(\eta) : \ldots]_J\) and \([\ldots : \hat x_J(x) : \ldots]_J\) coincide, where \(x = \varpi(\tilde x)\) and \(\eta = \eta(\tilde x)\).

The following proposition records the automorphy of the forms \(f_J\).

\begin{proposition}
	\label{prop:automorphy}
	For every \(g \in \Gamma_{M}\), every \((2,2,2,2)\)-partition \(J\), and every \(\eta \in \mathcal{B}\), we have
	\begin{align}
		f_J(g\eta)
		& = \chi\bigl(\jmath(g),\imath(\eta)\bigr)^2 f_J(\eta),
		\label{eq:fJ-automorphy}
		\\
		q_{U}(g\eta)
		& = \det(g)\,\chi\bigl(\jmath(g),\imath(\eta)\bigr)q_{U}(\eta),
		\label{eq:qU-automorphy}
	\end{align}
	where \(\jmath\) is the homomorphism of Fact~\ref{fact:iota}(2) and \(\chi(M,\tau) = \det(M_{21}\tau+M_{22})\) for \(M\in\Sp(12,\Q)\) and \(\tau\in\mathfrak{S}_6\).
	In particular, we have
	\[
		\frac{f_J(g\eta)}{q_{U}(g\eta)^2}
		= \frac{f_J(\eta)}{q_{U}(\eta)^2}.
	\]
\end{proposition}
\begin{proof}
	The forms \(f_J\) are defined in \cite[Section 2.4, Facts 3--5]{MMN07} as the translates of \(f_{J_1}\) under \(\Gamma_{U}\), and the same facts give
	\[
		f_J(g\eta) = \psi(g,\eta)^2 f_J(\eta),
		\qquad
		\psi(g,\eta) = \frac{q_{U}(g\eta)}{q_{U}(\eta)},
	\]
	for \(g\in \Gamma_{M}\).
	The quotient \(\psi(g,\eta)\) is well defined by Lemma~\ref{lem:nonvanishing}.
	This transformation law applies to each form \(f_J\) as a whole, and not to the individual theta constants occurring in \eqref{eq:fJ-list}.
	By \cite[Lemma 1]{MMN07}, we have \(\psi(g,\eta) = \det(g)\chi\bigl(\jmath(g),\imath(\eta)\bigr)\), which gives \eqref{eq:qU-automorphy}.
	Moreover, \cite[Proposition 2]{MMN07} gives \(\det(g)^2 = \operatorname{sign}(s(g)) = 1\) for \(g \in \Gamma_{M} = \ker s\), where \(s \colon \Gamma_{U} \to S_8\) is the homomorphism of \cite[Section 2.3]{MMN07} sending \(Tg_{jk}T^{-1}\) to the transposition \((j,k)\); the equality \(\det(g)^2 = 1\) yields \eqref{eq:fJ-automorphy}.
	Dividing \eqref{eq:fJ-automorphy} by the square of \eqref{eq:qU-automorphy} and using \(\det(g)^2=1\), we obtain the last assertion.
\end{proof}

\section{Constancy of the Ratio}
\label{sec:constancy}
In this section, we show that the ratio \(f_J(\eta)/\bigl(q_{U}(\eta)^2\, \hat x_J(x)\bigr)\) descends to a function on \(X(2,8)\) that is constant and independent of \(J\).

Throughout this section, we regard \(X(2,8)\) as a Zariski open subset of \(\Ps^5\) via the coordinates \((x_2,\ldots,x_6)\).
We use the homogeneous coordinates \([y_1 : y_2 : \cdots : y_6]\) on \(\Ps^5\), so that \(x_j = y_j / y_1\).
We write \(y\) for a point of \(\Ps^5\).
The boundary \(\Ps^5 \setminus X(2,8)\) is the union of the \(21\) hyperplanes
\begin{equation}
	\label{eq:boundary-list}
	\begin{gathered}
		D_{jk} = \{ y \in \Ps^5 \mid y_j = y_k \}
		\qquad (2 \leq j < k \leq 6),\\
		D_{0j} = \{ y \in \Ps^5 \mid y_j = 0 \},\qquad
		D_{1j} = \{ y \in \Ps^5 \mid y_j = y_1 \}
		\qquad (2 \leq j \leq 6),\\
		H_\infty = \{ y \in \Ps^5 \mid y_1 = 0 \},
	\end{gathered}
\end{equation}
that is, of the \(\binom{5}{2} + 5 + 5 = 20\) hyperplanes \(D_{\bullet\bullet}\) and the hyperplane at infinity.
By a generic point of a boundary hyperplane, we mean a point lying outside the intersections of that hyperplane with the other boundary hyperplanes.
A generic point of each of the \(20\) hyperplanes \(D_{\bullet\bullet}\) corresponds to the confluence of exactly two of the eight points, under the normalization \(x_1 = 0\), \(x_7 = 1\), and \(x_8 = \infty\).
In the coordinates \((x_2,\ldots,x_6)\), \(H_\infty\) is not given by any of the affine equations \(x_j = x_k\), \(x_j = 0\), or \(x_j = 1\).
A generic point of \(H_\infty\) corresponds to such a confluence only after a projective change of coordinates; see Step 3\('\) in the proof of Proposition~\ref{prop:constancy}.
The locus where more than two of the eight points coincide has codimension at least \(2\) in \(\Ps^5\), and so does the locus where each of two disjoint pairs of the eight points coincides.

We recall Deligne--Mostow's extension results, which apply near a generic point of each of the \(20\) hyperplanes \(D_{\bullet\bullet}\).

\begin{fact}
	\label{fact:DM86-extension}
	(\cite[Corollary~6.8, Proposition~8.7, and Sections~9.5--9.6]{DM86}) After we pass to a finite local cover and fix the marking, the period \(\eta\) extends holomorphically across each of the \(20\) hyperplanes \(D_{\bullet\bullet}\) in \eqref{eq:boundary-list} near a generic point of that hyperplane.
	The extension takes values in \(\mathcal{B}\).
\end{fact}

We next separate the covariance under changes of coordinates on \(\Ps^1\) from the equivariance under changes of the marking; the following lemma gives the covariance.

\begin{lemma}
	\label{lem:pgl-covariance}
	Let \(p_1,\ldots,p_8\) be pairwise distinct points of \(\Ps^1\), written \(p_l=[u_l:v_l]\), and put
	\[
		[p_l,p_m]=u_lv_m-u_mv_l,\qquad
		P_J=\prod_{k=1}^4[p_{j_{2k-1}},p_{j_{2k}}],
		\qquad
		\omega=\frac{v\,du-u\,dv}{w},
	\]
	where \(p=[u:v]\) and \(w^4=\prod_{l=1}^8[p,p_l]\).
	Here \(\eta\) denotes the period of \(\omega\) with respect to the cycles \(\mathfrak{a}_1,\ldots,\mathfrak{a}_6\) of Fact~\ref{fact:MT04-periods}.
	We transport these cycles by the projective change of coordinates that carries the normalized representative of the class of \((p_1,\ldots,p_8)\) to \((p_1,\ldots,p_8)\).
	We define
	\[
		G_J(p_1,\ldots,p_8)
		=\frac{f_J(\eta)}{q_{U}(\eta)^2\,\varepsilon_JP_J}.
	\]
	Then \(G_J\) is independent of the lifts \((u_l,v_l)\) of the points \(p_l\) and invariant under \(\PGL(2,\C)\).
	In the affine chart \(p_l=[x_l:1]\), \(p_8=[1:0]\), this definition gives
	\[
		G_J=\frac{f_J(\eta)}{q_{U}(\eta)^2\,\hat x_J(x)}.
	\]
\end{lemma}
\begin{proof}
	Suppose that \(p_l\) is replaced by \(t_lp_l\) with \(t_l \in \C^{\times}\) for \(1\leq l\leq 8\), and put \(R=\prod_{l=1}^8t_l\).
	Then \(P_J\) is multiplied by \(R\), independently of \(J\).
	We choose a local fourth root of \(R\).
	Then \(w\) is multiplied by \(R^{1/4}\), so \(\omega\) and \(\eta\) are multiplied by \(R^{-1/4}\), and \(q_{U}(\eta)^2\) is multiplied by \(R^{-1}\).
	By Fact~\ref{fact:iota}(1), the matrix \(\imath(\eta)\), and therefore \(f_J(\eta)\), is unchanged by scalar multiplication of \(\eta\).
	Thus \(G_J\) is independent of the lifts.
	A different choice of the fourth root only multiplies \(\eta\) by a fourth root of unity and yields the same conclusion.

	For \(A\in\GL(2,\C)\), we replace \(p_l\) by \(Ap_l\) for \(1\leq l\leq 8\), and we have
	\[
		[Ap_l,Ap_m]=\det(A)[p_l,p_m],
	\]
	so that \(P_J\) is multiplied by \(\det(A)^4\).
	The numerator \(v\,du-u\,dv\) is multiplied by \(\det(A)\), and the product \(\prod_{l=1}^8[p,p_l]\) is multiplied by \(\det(A)^8\).
	Therefore \(\omega\) and \(\eta\) are multiplied by \(\det(A)^{-1}\), and \(q_{U}(\eta)^2\) is multiplied by \(\det(A)^{-4}\).
	Thus the two factors in the denominator of \(G_J\) are multiplied by \(\det(A)^{-4}\) and \(\det(A)^{4}\), respectively, but \(f_J(\eta)\) is unchanged.
	Hence \(G_J\) is invariant under \(\PGL(2,\C)\).

	In the affine chart of the statement, we write \(w_h\) for the function of the homogeneous model, whose equation is
	\[
		w_h^4=-\prod_{j=1}^7(z-x_j).
	\]
	The affine model used in Section~\ref{sec:preliminaries} is \(w_a^4=\prod_{j=1}^7(z-x_j)\).
	We choose a primitive eighth root of unity \(\zeta_8\), so that \(\zeta_8^4=-1\), and identify the two models by \(w_h=\zeta_8 w_a\).
	Then \(\eta_h=\zeta_8^{-1}\eta_a\) and
	\[
		q_{U}(\eta_h)^2=\zeta_8^{-4}q_{U}(\eta_a)^2=-q_{U}(\eta_a)^2.
	\]
	Since \(P_J=-x_J\) and \(f_J(\eta)\) is unchanged by scalar multiplication of \(\eta\), the two negative signs cancel and give the affine formula in the statement.
\end{proof}

The following lemma gives the equivariance under changes of the marking.

\begin{lemma}
	\label{lem:gamma-equivariance}
	Let \(g \in \Gamma_{U}\), and put \(\sigma = s(g) \in S_8\), where \(s \colon \Gamma_{U} \to S_8\) is the homomorphism of \cite[Section 2.3]{MMN07}.
	Then, for every \((2,2,2,2)\)-partition \(J\), there is a sign \(\varepsilon(\sigma,J) \in \{\pm1\}\) such that
	\[
		\hat x_J(\sigma\cdot x) = \varepsilon(\sigma,J)\,\hat x_{\sigma^{-1}(J)}(x),
		\qquad
		f_J(g\eta) = \varepsilon(\sigma,J)\,\psi(g,\eta)^2 f_{\sigma^{-1}(J)}(\eta),
	\]
	where \(\psi(g,\eta) = q_{U}(g\eta)/q_{U}(\eta)\).
	In particular, we have
	\[
		\frac{f_J(g\eta)}{q_{U}(g\eta)^2\,\hat x_J(\sigma\cdot x)}
		=\frac{f_{\sigma^{-1}(J)}(\eta)}{q_{U}(\eta)^2\,\hat x_{\sigma^{-1}(J)}(x)}.
	\]
\end{lemma}
\begin{proof}
	Reordering the two entries in each factor of \(x_J(\sigma\cdot x) = \pm x_{\sigma^{-1}(J)}(x)\) and inserting the signs \(\varepsilon_J\) and \(\varepsilon_{\sigma^{-1}(J)}\) of \(\hat x_J = \varepsilon_J x_J\), we obtain the first identity, with a sign \(\varepsilon(\sigma,J)\) depending only on \(\sigma\) and \(J\).
	For the second identity, we use the action \(f^{g}(\eta) = \psi(g,\eta)^{-2} f(g\eta)\) of \(\Gamma_{U}\) on the automorphic forms of weight \(2\) with respect to \(\Gamma_{M}\); cf.\ \cite[Section 2.4]{MMN07}.
	By \cite[Fact 5]{MMN07} and the normalization of \cite[Section 5]{MMN07}, the collections \((f_J)_J\) and \((\hat x_J)_J\) are equivariant under \(S_8\) with the same multipliers; that is, \(f_J^{g} = \varepsilon(\sigma,J)\, f_{\sigma^{-1}(J)}\) with the sign \(\varepsilon(\sigma,J)\) of the first identity.
	This equality is the second identity.
	Since every coordinate \(\hat x_J(x)\) is nonzero on \(X(2,8)\), the equality of the two projective points established in \cite[Fact 5]{MMN07} shows that every coordinate \(f_J(\eta)\) is nonzero, and the quotients in the last assertion are defined.
	Dividing the second identity by \(q_{U}(g\eta)^2 = \psi(g,\eta)^2\,q_{U}(\eta)^2\) and by the first identity, we obtain the last assertion.
\end{proof}

Then we have the following proposition.

\begin{proposition}
	\label{prop:constancy}
	For each \((2,2,2,2)\)-partition \(J\), the function \(G_J\) of Lemma~\ref{lem:pgl-covariance} descends from \(\widetilde{X}(2,8)\) to a single-valued holomorphic function on \(X(2,8)\) and is constant.
	Moreover, the constant is independent of \(J\); we write \(G\) for the common function.
\end{proposition}
\begin{proof}
	\emph{Step 1: single-valuedness.}
	For \(\xi\in\mathcal{B}\), we set \(F_J(\xi)=f_J(\xi)/q_{U}(\xi)^2\).
	By Lemma~\ref{lem:nonvanishing}, the denominator does not vanish, so \(F_J\) is holomorphic.
	Proposition~\ref{prop:automorphy} establishes the cancellation of the factor \(\psi(g,\eta)^2\) and shows that \(F_J\) is invariant under \(\Gamma_{M}\).
	Thus the function \(\tilde x\mapsto F_J(\eta(\tilde x))\), and therefore \(G_J\), is invariant under the deck transformations of \(\varpi\) and descends to a single-valued holomorphic function on \(X(2,8)\).

	\emph{Step 2: independence of \(J\).}
	As recalled in Section~\ref{sec:preliminaries}, the forms \(f_J\) in \eqref{eq:fJ-list} represent the forms \(T_J^{(2)}\) of \cite{MT04} in the normalization of \cite[Section 5]{MMN07}.
	By \cite[Fact 5]{MMN07}, which rests on \cite[Theorem 6.4]{MT04}, the projective points \([\ldots : f_J(\eta) : \ldots]_J\) and \([\ldots : \hat x_J(x) : \ldots]_J\) coincide.
	Since every coordinate \(\hat x_J(x)\) is nonzero on \(X(2,8)\), the equality of the two projective points shows that the ratio \(f_J(\eta)/\hat x_J(x)\) is independent of \(J\).
	Dividing the ratio \(f_J(\eta)/\hat x_J(x)\) by the common factor \(q_{U}(\eta)^2\), which is nonzero by Lemma~\ref{lem:nonvanishing}, we obtain \(G_J = G_{J'}\) on \(X(2,8)\) for all \(J\) and \(J'\).

	\emph{Step 3: extension across the hyperplanes \(D_{\bullet\bullet}\).}
	Let \(D\) be one of the \(20\) hyperplanes \(D_{\bullet\bullet}\) in \eqref{eq:boundary-list}, and let \(x_0\) be a generic point of \(D\).
	We choose a partition \(J'\) that does not pair the two points coinciding on \(D\), so that \(\hat x_{J'}(x_0) \neq 0\).
	By Fact~\ref{fact:DM86-extension}, the period \(\eta\) extends holomorphically across the preimage of \(D\), with values in \(\mathcal{B}\), after we pass to a finite local cover and fix the marking.
	Since \(F_{J'}\) is holomorphic on \(\mathcal{B}\) and \(\hat x_{J'}(x_0)\neq0\), the function \(F_{J'}(\eta)/\hat x_{J'}\) extends holomorphically to this local cover.
	By Step 1, away from the preimage of \(D\), the extension of \(F_{J'}(\eta)/\hat x_{J'}\) is invariant under the deck transformations of this local cover, and hence the extension is invariant everywhere by the identity theorem.
	The extension of \(F_{J'}(\eta)/\hat x_{J'}\) therefore descends to a holomorphic extension of \(G\) across \(D\) at \(x_0\).

	\emph{Step 3\('\): extension across the hyperplane at infinity.}
	We change the coordinate on \(\Ps^1\) by \(\Ps^1 \ni z \mapsto y_1z/y_2 \in \Ps^1\) and then transpose the labels \(2\) and \(7\).
	The composite takes the configuration \((0,x_2,\ldots,x_6,1,\infty)\) with \(x_j = y_j/y_1\) to the configuration \((0,x_2',\ldots,x_6',1,\infty)\) with
	\[
		x_2' = \frac{y_1}{y_2},
		\qquad
		x_j' = \frac{y_j}{y_2}
		\qquad (3\leq j\leq 6).
	\]
	In the homogeneous coordinates of \eqref{eq:boundary-list}, the composite is the linear map
	\[
		\sigma_\infty\colon
		\Ps^5 \ni [y_1:y_2:y_3:y_4:y_5:y_6]\longmapsto[y_2:y_1:y_3:y_4:y_5:y_6] \in \Ps^5,
	\]
	an involutive automorphism of \(\Ps^5\) that carries \(H_\infty\) onto \(D_{02}\).
	The map \(\sigma_\infty\) carries a generic point of \(H_\infty\) to a generic point of \(D_{02}\).
	Moreover, \(\sigma_\infty\) restricts to a biholomorphism from a neighborhood of that generic point of \(H_\infty\) onto a neighborhood of the image of that point.
	The transposition of the labels \(2\) and \(7\) is induced by the unitary reflection \(g_{27}\) of Fact~\ref{fact:MT04-periods}, and \(Tg_{27}T^{-1}\) lies in \(\Gamma_{U}\) and satisfies \(s(Tg_{27}T^{-1}) = (2\,7)\).
	Hence, by Lemmas~\ref{lem:pgl-covariance} and~\ref{lem:gamma-equivariance}, together with the equality \(G_J = G_{J'}\) established in Step 2, we have \(G\circ\sigma_\infty = G\).
	By Step 3, \(G\) extends holomorphically across \(D_{02}\) at a generic point, so \(G\) extends holomorphically across \(H_\infty\) at a generic point as well.

	\emph{Step 4: extension to \(\Ps^5\) and constancy.}
	We write \(A\) for the union of the pairwise intersections of the \(21\) boundary hyperplanes in \eqref{eq:boundary-list}.
	The set \(A\) is a finite union of linear subspaces of codimension \(2\) in \(\Ps^5\); in particular, \(A\) is a closed analytic subset of codimension \(2\).
	Every point of \(\Ps^5 \setminus A\) lies in \(X(2,8)\) or is a generic point of exactly one boundary hyperplane.
	By Steps 1, 3, and 3\('\), \(G\) extends holomorphically across every boundary hyperplane at every generic point of that hyperplane.
	Each local extension is unique because it agrees with \(G\) on the dense open intersection of its domain with \(X(2,8)\).
	Therefore the identity theorem shows that these extensions glue to a holomorphic function on \(\Ps^5 \setminus A\).
	By Hartogs' theorem, \(G\) extends holomorphically to \(\Ps^5\).
	Since \(\Ps^5\) is compact and connected, every global holomorphic function on it is constant.
	Hence \(G\) is constant.
\end{proof}

\begin{remark}
	\label{rem:marking}
	If we change the lift \(\tilde x\) by a deck transformation, then \(\eta\) changes by an element of \(\Gamma_{M}\).
	By Proposition~\ref{prop:constancy}, this change does not affect the \ref{thm:main}.
	A change of the marking by an element of \(\Gamma_{U}\) can also permute the forms \(f_J\) and the polynomials \(\hat x_J\).
	By Lemma~\ref{lem:gamma-equivariance}, the collection of identities \eqref{eq:thomae} is preserved under such a change.
\end{remark}

By Proposition~\ref{prop:constancy}, the function \(G\) is a constant, and we write \(\kappa\) for its value.
We have only to evaluate \(G\) at one degenerate configuration.
We use the auxiliary partition
\[
	J^{\ast} = \langle 14;25;36;78 \rangle,
	\qquad
	x_{J^\ast} = (x_1-x_4)(x_2-x_5)(x_3-x_6).
\]
The corresponding automorphic form is
\begin{equation}
	\label{eq:fJstar}
	f_{J^\ast} =
	\left({\vt{111100}{000000}}^2 + {\vt{110000}{001100}}^2\right)
	\left({\vt{000011}{111111}}^2 + {\vt{001111}{110011}}^2\right)
\end{equation}
by entry \#15 of the table in \cite[Theorem 2]{MMN07}.
In that table, \(f_{J^\ast}\) is the entry labeled \(g_3\) with \(\varepsilon = \varepsilon' = +1\).
It is attached to \(+x_{\langle 14;25;36;78\rangle}\) and has the same four characteristics as \(f_{J_3}\); hence \(\hat x_{J^\ast} = x_{J^\ast}\).

\section{Degeneration and Splitting of the Theta Constants}
\label{sec:degeneration}
In this section, we compute the period \(\eta\) at a degenerate configuration, express the constant \(\kappa\) in terms of the value of \(f_{J^\ast}\) there, and split the theta constants into theta constants on \(\mathbb{H}\).

\subsection{The Degenerate Configuration}

We fix \(c \in (0,1)\) and consider the degenerate configuration
\[
	(0,\,0,\,0,\,c,\,c,\,c,\,1,\,\infty),
\]
the limit of \(x \in X(2,8)\) as \(x_2, x_3 \to 0\) and \(x_4, x_5, x_6 \to c\).
The normalization of the limit curve \(w^4 = z^3(z-c)^3(z-1)\), that is, the smooth projective model of this curve, has genus \(3\).
In fact, the map from the normalization to \(\Ps^1\) is of degree four and totally ramified over \(0,c,1,\infty\), so the Riemann--Hurwitz formula gives \(2g-2=4(-2)+4(4-1)=4\) for the genus \(g\) of the normalization, and hence \(g=3\).
Each factor \(z - x_j\) contributes the exponent \(\frac14\) to \(\frac{dz}{w}\), so each of the two triples of coinciding branch points contributes the total exponent \(3\cdot\frac14 = \frac34 < 1\); hence all period integrals converge.
Therefore \(\eta\) extends continuously to the degenerate configuration along the closure of \(X_{\R}\), and we write \(\eta_c\) for the limit value.
We use only this one-sided continuity here, since Proposition~\ref{prop:constancy} shows that the function \(G\) is constant on \(X(2,8)\).

Then we have the following proposition.

\begin{proposition}
	\label{prop:eta-c}
	We have
	\begin{equation}
		\label{eq:eta-c}
		\eta_c = \Transpose{(0,\,0,\,v,\,v,\,v,\,u)},
		\qquad
		v = 2i\mathcal{K}_1,\quad u = 2i(\mathcal{K}_1 - \sqrt{2}\mathcal{K}_2),
	\end{equation}
	where
	\begin{equation}
		\label{eq:K-hypergeometric}
		\begin{gathered}
			\mathcal{K}_1 = \mathcal{K}_1(c) = \int_0^c \frac{dz}{\left( z^3(c-z)^3(1-z) \right)^{1/4}}
			= \frac{\varGamma(1/4)^2}{\sqrt{\pi c}}\,
			F\!\left( \tfrac14,\tfrac14;\tfrac12; c \right),\\
			\mathcal{K}_2 = \mathcal{K}_2(c) = \int_c^1 \frac{dz}{\left( z^3(z-c)^3(1-z) \right)^{1/4}}
			= \sqrt{2}\,\pi\,
			F\!\left( \tfrac34,\tfrac34;1; 1-c \right).
		\end{gathered}
	\end{equation}
	Here \(F\) denotes the Gauss hypergeometric series.
	In both integrands we take the positive real fourth root.
	In particular, we have
	\[
		\Transpose{\eta_c}U\eta_c = u^2 - v^2
		= 8\mathcal{K}_2(\sqrt2 \mathcal{K}_1 - \mathcal{K}_2)
		= -h_{U}(\eta_c, \eta_c).
	\]
\end{proposition}
\begin{proof}
	In the limit, the intervals \(\ell_1, \ell_2, \ell_4, \ell_5\) shrink to points and \(\mathcal{I}_1 = \mathcal{I}_2 = \mathcal{I}_4 = \mathcal{I}_5 = 0\).
	In the same limit, the branch of \(w\) fixed in Section~\ref{sec:preliminaries} gives \(\mathcal{I}_3 = -\mathcal{K}_1\) on \(\ell_3 = [0,c]\) and \(\mathcal{I}_6 = \exp(-\pi i/4)\mathcal{K}_2\) on \(\ell_6 = [c,1]\).
	As in Section~\ref{sec:preliminaries}, the chains of Fact~\ref{fact:MT04-periods} satisfy \(\int_{A_k} \frac{dz}{w} = 2\,\mathcal{I}_k\) and \(\int_{B_k} \frac{dz}{w} = -2i\,\mathcal{I}_k\).
	Thus the cycles of Fact~\ref{fact:MT04-periods} give \(\eta_1 = \eta_2 = 0\), \(\eta_3 = \eta_4 = \eta_5 = -2i\mathcal{I}_3 = v\), and \(\eta_6 = -2i\mathcal{I}_3 + 2(1-i)\mathcal{I}_6 = u\).
	The expressions \eqref{eq:K-hypergeometric} follow from the Euler integral representation of the hypergeometric series, together with \(\varGamma(3/4)\varGamma(1/4) = \sqrt{2}\,\pi\).
	Since \(u\) and \(v\) are purely imaginary, we have \(\Transpose{\eta_c}U\eta_c = -h_{U}(\eta_c,\eta_c)\), and \eqref{eq:eta-c} gives the remaining equalities in the statement.
\end{proof}

We set
\[
	m=\frac{1-\sqrt{c}}{2}\in(0,\tfrac12),
	\qquad
	\mathcal{F}(z)=F(\tfrac12,\tfrac12;1;z).
\]
Then we have the following lemma.

\begin{lemma}
	\label{lem:eta-c-positivity}
	We have
	\begin{equation}
		\label{eq:K-m}
		\mathcal{K}_2=\frac{\sqrt2\,\pi}{\sqrt{c}}\mathcal{F}(m),
		\qquad
		\sqrt2\mathcal{K}_1
		=\frac{\sqrt2\,\pi}{\sqrt{c}}\bigl(\mathcal{F}(m)+\mathcal{F}(1-m)\bigr).
	\end{equation}
	In particular, \(\sqrt2\mathcal{K}_1>\mathcal{K}_2\) and \(\Transpose{\eta_c}U\eta_c>0\).
\end{lemma}
\begin{proof}
	We use Euler's transformation and the quadratic transformation given in \cite[Section~2.9, formula~(2), and Section~2.11, formula~(2)]{EMOT53}.
	We apply Euler's transformation with parameters \((\tfrac34,\tfrac34;1)\) and the quadratic transformation with \(a=b=\tfrac14\), so that \(a+b+\tfrac12=1\).
	Using \(1-c=4m(1-m)\) and the positive real branches, we obtain
	\[
		F\!\left(\tfrac34,\tfrac34;1;1-c\right)
		=\frac{\mathcal{F}(m)}{1-2m}
		=\frac{\mathcal{F}(m)}{\sqrt{c}}.
	\]
	The expression for \(\mathcal{K}_2\) in \eqref{eq:K-hypergeometric} gives the first equality in \eqref{eq:K-m}.

	We also use the connection formula
	\begin{equation}
		\label{eq:T2}
		F\!\left(\tfrac14,\tfrac14;\tfrac12;(1-2m)^2\right)
		=\frac{\pi^{3/2}}{\varGamma(1/4)^2}
		\bigl(\mathcal{F}(m)+\mathcal{F}(1-m)\bigr),
		\qquad 0<m<1.
	\end{equation}
	To verify \eqref{eq:T2}, we use the identity
	\[
		\mathcal{F}\!\left(\frac{1-z}{2}\right)
		=\frac{\sqrt{\pi}}{\varGamma(3/4)^2}
		F\!\left(\tfrac14,\tfrac14;\tfrac12;z^2\right)
		-\frac{2\sqrt{\pi}}{\varGamma(1/4)^2}
		zF\!\left(\tfrac34,\tfrac34;\tfrac32;z^2\right),
	\]
	which follows from \cite[Section~2.11, formula~(3)]{EMOT53} by setting \(a=b=\tfrac14\) and replacing \(z\) by \(-z\).
	We apply this identity at \(z=1-2m\) and at \(z=2m-1\), and add the resulting identities.
	For \(-1<z<1\), all branches in this identity are obtained by analytic continuation of the defining series along the real interval from \(z=0\), and none of their arguments crosses the standard cut \([1,\infty)\).
	The terms that are odd in \(z\) cancel, and the identity \(\varGamma(3/4)\varGamma(1/4)=\sqrt{2}\,\pi\) gives \eqref{eq:T2}.
	Substituting \eqref{eq:T2} into the expression for \(\mathcal{K}_1\) in \eqref{eq:K-hypergeometric}, we obtain the second equality in \eqref{eq:K-m}.
	Since \(\mathcal{F}(m)\) and \(\mathcal{F}(1-m)\) are positive, we obtain \(\mathcal{K}_2>0\) and \(\sqrt2\mathcal{K}_1>\mathcal{K}_2\) from \eqref{eq:K-m}.
	The identity \(\Transpose{\eta_c}U\eta_c = 8\mathcal{K}_2(\sqrt2 \mathcal{K}_1 - \mathcal{K}_2)\) of Proposition~\ref{prop:eta-c} then gives \(\Transpose{\eta_c}U\eta_c>0\).
\end{proof}

\subsection{Reduction to a Single Value}

By Lemma~\ref{lem:eta-c-positivity} and Fact~\ref{fact:iota}(1), the matrix \(\imath(\eta_c)\) lies in \(\mathfrak{S}_6\), and the theta constants and \(f_{J^\ast}\) also extend continuously to the degenerate configuration.
Since the first two coordinates of \(\eta_c\) vanish, a direct computation from the definition of \(\imath\) shows that \(\imath(\eta_c)\) is block diagonal:
\begin{equation*}
	\imath(\eta_c) = iI_2 \oplus \tau_c,
	\qquad
	\tau_c = -\frac{i}{u^2-v^2}
	\begin{pmatrix}
		2v^2    & u^2+v^2 & 2v^2     & 2uv     \\
		u^2+v^2 & 2v^2    & 2v^2     & 2uv     \\
		2v^2    & 2v^2    & 3v^2-u^2 & 2uv     \\
		2uv     & 2uv     & 2uv      & u^2+v^2
	\end{pmatrix}.
\end{equation*}
Since \(\imath(\eta_c)\) lies in \(\mathfrak{S}_6\) and is block diagonal, each of its diagonal blocks has positive definite imaginary part.
In particular, the block \(\tau_c\) lies in \(\mathfrak{S}_4\).
Every factor of \(G\) therefore extends continuously to the degenerate configuration, and \(x_{J^\ast}(0,0,0,c,c,c,1,\infty) = (0-c)^3 = -c^3 \neq 0\).
Letting \(x\) tend to this configuration in the identity \(G = \kappa\) of Section~\ref{sec:constancy}, we obtain
\begin{equation}
	\label{eq:kappa-anchor}
	\kappa = \frac{f_{J^\ast}(\eta_c)}{(\Transpose{\eta_c}U\eta_c)^2\,(-c^3)}.
\end{equation}
In the rest of this section and in Section~\ref{sec:evaluation}, we compute the numerator and the denominator of \eqref{eq:kappa-anchor}.

\subsection{Splitting of the Theta Constants}

The defining series of each theta constant splits along the block decomposition \(\imath(\eta_c) = iI_2 \oplus \tau_c\).
For every characteristic, writing \(a = (a_1, a_2, a')\) and \(b = (b_1,b_2,b')\) with \(a', b' \in \Z^4\), we have
\begin{equation}
	\label{eq:theta-split-trivial}
	\vt{a}{b}(\eta_c)
	= \vartheta_{a_1b_1}(i)\,\vartheta_{a_2b_2}(i)\,\vt{a'}{b'}(\tau_c).
\end{equation}
Applying \eqref{eq:theta-split-trivial} to the four characteristics of \eqref{eq:fJstar} and using \(\vartheta_{10}(i)^2 = \vartheta_{01}(i)^2 = \vartheta_{00}(i)^2/\sqrt{2}\), which follows from \eqref{eq:classical-4} below, we obtain
\begin{equation}
	\label{eq:fJstar-blocked}
	f_{J^\ast}(\eta_c)
	= \frac{\vartheta_{00}(i)^{8}}{4}
	\left( \Theta_1^2 + \Theta_2^2 \right)
	\left( \Theta_3^2 + \Theta_4^2 \right),
\end{equation}
where \(\Theta_1 = \vt{1100}{0000}(\tau_c)\), \(\Theta_2 = \vt{0000}{1100}(\tau_c)\), \(\Theta_3 = \vt{0011}{1111}(\tau_c)\), and \(\Theta_4 = \vt{1111}{0011}(\tau_c)\).

We set \(\tau^\ast = -i(u+v)/(u-v)\), with \(u\) and \(v\) as in Proposition~\ref{prop:eta-c}.
By Lemma~\ref{lem:eta-c-positivity}, we have \(\tau^\ast,-1/\tau^\ast\in i\R_{>0}\).
We use these positive imaginary values to fix the branches of the square roots in \eqref{eq:classical-3}, as stated after \eqref{eq:classical-5} below.

Then we have the following proposition.

\begin{proposition}
	\label{prop:splitting}
	For \(n=(n_1,n_2,n_3,n_4)\in\Z^4\), we set
	\[
		r=(r_1,r_2,r_3,r_4)=n\Transpose{L},
		\qquad
		L=
		\begin{pmatrix}
			1 & 1  & 1 & 1  \\
			1 & 1  & 1 & -1 \\
			1 & 1  & 2 & 0  \\
			1 & -1 & 0 & 0
		\end{pmatrix},
		\qquad
		\det L=-4.
	\]
	Then \(n\mapsto r\) is a bijection of \(\Z^4\) onto the sublattice of index \(4\)
	\[
		\{r\in\Z^4\mid r_1\equiv r_2,\ r_3\equiv r_4\pmod 2\},
	\]
	with inverse
	\[
		\begin{aligned}
			n_1 &= \frac{r_1+r_2-r_3+r_4}{2}, &
			n_2 &= \frac{r_1+r_2-r_3-r_4}{2}, \\
			n_3 &= r_3 - \frac{r_1+r_2}{2}, &
			n_4 &= \frac{r_1-r_2}{2}.
		\end{aligned}
	\]
	Moreover, we have
	\[
		\frac{1}{2}\,n\,\tau_c\,\Transpose{n}
		= \frac{\tau^\ast}{4}r_1^2-\frac{1}{4\tau^\ast}r_2^2+\frac{i}{4}r_3^2+\frac{i}{4}r_4^2.
	\]
	For \(a,b\in\Z^4\), we set \(\alpha=a\Transpose{L}\) and \(\beta=bL^{-1}\).
	The components of \(\alpha\) and \(\beta\) are
	\begin{align*}
		(\alpha_1,\alpha_2,\alpha_3,\alpha_4)
		 & = (a_1{+}a_2{+}a_3{+}a_4,\ a_1{+}a_2{+}a_3{-}a_4,\ a_1{+}a_2{+}2a_3,\ a_1{-}a_2), \\
		(\beta_1,\beta_2,\beta_3,\beta_4)
		 & = \left( \tfrac{b_1+b_2-b_3+b_4}{2},\ \tfrac{b_1+b_2-b_3-b_4}{2},\
		b_3 - \tfrac{b_1+b_2}{2},\ \tfrac{b_1-b_2}{2} \right).
	\end{align*}
	Consequently, we have
	\begin{equation}
		\label{eq:general-split}
		\begin{aligned}
			\vt{a}{b}(\tau_c)
			= \frac{1}{4}\sum_{\nu_1,\nu_2 \in \{0,1\}}
			& \e\left( -\frac{(\alpha_1{+}\alpha_2)\nu_1 + (\alpha_3{+}\alpha_4)\nu_2}{4} \right)
			\vartheta_{\alpha_1,\beta_1{+}\nu_1}\!\left(\tfrac{\tau^\ast}{2}\right) \\
			& \times
			\vartheta_{\alpha_2,\beta_2{+}\nu_1}\!\left(-\tfrac{1}{2\tau^\ast}\right)
			\vartheta_{\alpha_3,\beta_3{+}\nu_2}\!\left(\tfrac{i}{2}\right)
			\vartheta_{\alpha_4,\beta_4{+}\nu_2}\!\left(\tfrac{i}{2}\right).
		\end{aligned}
	\end{equation}
\end{proposition}
\begin{proof}
	Substituting the expressions for \(n_1,n_2,n_3,n_4\) into the quadratic form, we obtain
	\[
		\frac{1}{2i}\,n\,\tau_c\,\Transpose{n}
		=-\frac{u+v}{4(u-v)}r_1^2-\frac{u-v}{4(u+v)}r_2^2+\frac{r_3^2}{4}+\frac{r_4^2}{4}.
	\]
	The definition of \(\tau^\ast\) gives the claimed quadratic identity.
	Every element of the image of \(n\mapsto r\) satisfies \(r_1\equiv r_2\) and \(r_3\equiv r_4\pmod 2\), since \(r_1-r_2=2n_4\) and \(r_3-r_4=2(n_2+n_3)\).
	Conversely, suppose that \(r_1\equiv r_2\) and \(r_3\equiv r_4\pmod 2\).
	Then the four displayed expressions for \(n_1,n_2,n_3,n_4\) are integers, and substituting them back, we recover \(r\), which proves that \(n\mapsto r\) is a bijection.
	Since \(r+\alpha/2=(n+a/2)\Transpose{L}\) and \(\beta=bL^{-1}\), the defining series becomes
	\[
		\begin{aligned}
			\vt{a}{b}(\tau_c)
			&=\sum_{\substack{r\in\Z^4\\ r_1\equiv r_2\pmod 2\\ r_3\equiv r_4\pmod 2}}
			\e\Bigg(
			\frac{\tau^\ast}{4}\left(r_1+\frac{\alpha_1}{2}\right)^2
			-\frac{1}{4\tau^\ast}\left(r_2+\frac{\alpha_2}{2}\right)^2 \\
			&\qquad
			+\frac{i}{4}\sum_{j=3}^4\left(r_j+\frac{\alpha_j}{2}\right)^2
			+\frac12\sum_{j=1}^4\left(r_j+\frac{\alpha_j}{2}\right)\beta_j
			\Bigg).
		\end{aligned}
	\]
	For integers \(r_1,r_2,r_3,r_4\), we have
	\[
		\begin{aligned}
			\frac12\sum_{\nu_1\in\{0,1\}}\e\left(\frac{\nu_1(r_1+r_2)}{2}\right)&=\frac12\left(1+(-1)^{r_1+r_2}\right), \\
			\frac12\sum_{\nu_2\in\{0,1\}}\e\left(\frac{\nu_2(r_3+r_4)}{2}\right)&=\frac12\left(1+(-1)^{r_3+r_4}\right).
		\end{aligned}
	\]
	The sum over \(\nu_1\) is \(1\) if \(r_1\equiv r_2\pmod 2\) and \(0\) otherwise, and the sum over \(\nu_2\) has the same property for \(r_3\) and \(r_4\).
	Multiplying the summand by these two sums, we replace the constrained sum by a sum over all \(r\in\Z^4\) and obtain the factor \(1/4\).
	For fixed \(\nu_1,\nu_2\in\{0,1\}\), the sum over \(r\in\Z^4\) factors into four sums over \(r_1\), \(r_2\), \(r_3\), and \(r_4\).
	For example, the sum over \(r_1\) is
	\[
		\begin{aligned}
			&\sum_{r_1\in\Z}
			\e\left(
			\frac{\tau^\ast}{4}\left(r_1+\frac{\alpha_1}{2}\right)^2
			+\frac{\beta_1}{2}\left(r_1+\frac{\alpha_1}{2}\right)
			+\frac{\nu_1 r_1}{2}
			\right) \\
			&\qquad
			=\e\left(-\frac{\alpha_1\nu_1}{4}\right)
			\vartheta_{\alpha_1,\beta_1+\nu_1}\!\left(\frac{\tau^\ast}{2}\right).
		\end{aligned}
	\]
	Here \(\nu_1 r_1/2=\nu_1(r_1+\alpha_1/2)/2-\alpha_1\nu_1/4\), which produces the factor \(\e(-\alpha_1\nu_1/4)\).
	The sums over \(r_2\), \(r_3\), and \(r_4\) give the remaining three theta constants in \eqref{eq:general-split}, together with the factors \(\e(-\alpha_2\nu_1/4)\), \(\e(-\alpha_3\nu_2/4)\), and \(\e(-\alpha_4\nu_2/4)\), respectively.
	The product of the four factors is
	\[
		\e\left(-\frac{(\alpha_1+\alpha_2)\nu_1+(\alpha_3+\alpha_4)\nu_2}{4}\right),
	\]
	which gives \eqref{eq:general-split}.
\end{proof}

Hereafter, we use the following classical identities for Jacobi's theta constants; cf.\ \cite[Chapter 2]{BB98}.
We define the elliptic modular function \(\lambda\) by
\[
	\lambda(\tau) = \frac{\vartheta_{10}(\tau)^4}{\vartheta_{00}(\tau)^4},
	\qquad \tau \in \mathbb{H}.
\]
The function \(\lambda\) restricts to a bijection from \(i\R_{>0}\) onto \((0,1)\).
For \(\tau \in \mathbb{H}\), we have
\begin{gather}
	\label{eq:classical-1}
	\vartheta_{00}(\tau) = \vartheta_{00}(4\tau) + \vartheta_{10}(4\tau),
	\qquad
	\vartheta_{01}(\tau) = \vartheta_{00}(4\tau) - \vartheta_{10}(4\tau),
	\\
	\label{eq:classical-2}
	\vartheta_{01}(\tau)^2 = \vartheta_{00}(2\tau)^2 - \vartheta_{10}(2\tau)^2,
	\qquad
	\vartheta_{10}(\tau)^2 = 2\,\vartheta_{00}(2\tau)\vartheta_{10}(2\tau),
	\\
	\label{eq:classical-3}
	\vartheta_{00}(-1/\tau) = \sqrt{-i\tau}\;\vartheta_{00}(\tau),
	\qquad
	\vartheta_{01}(-1/\tau) = \sqrt{-i\tau}\;\vartheta_{10}(\tau),
	\\
	\label{eq:classical-4}
	\vartheta_{00}(i) = \frac{\pi^{1/4}}{\varGamma(3/4)},
	\qquad
	\vartheta_{01}(i) = \vartheta_{10}(i) = 2^{-1/4}\,\vartheta_{00}(i),
	\\
	\label{eq:classical-5}
	\vartheta_{00}(\tau)^4 = \vartheta_{01}(\tau)^4 + \vartheta_{10}(\tau)^4,
	\qquad
	\vartheta_{00}(\tau)^2 = F\!\left(\tfrac12,\tfrac12;1;\lambda(\tau)\right),
\end{gather}
where \(\sqrt{-i\tau}\) denotes the principal branch, which is real and positive for \(\tau \in i\R_{>0}\).
In the second identity of \eqref{eq:classical-5}, the hypergeometric function is continued from its defining series at \(\lambda(i)=1/2\) along the image under \(\lambda\) of a path in \(\mathbb{H}\) from \(i\) to \(\tau\).
All the moduli to which \eqref{eq:classical-3} is applied below lie in \(i\R_{>0}\).
In \eqref{eq:classical-5}, the two identities are due to Jacobi.
We obtain \eqref{eq:classical-1} by separating the even and odd terms in the defining series; cf.\ \cite[Section~2.1, equations~(2.1.1)--(2.1.4)]{BB98}.
The first identities in \eqref{eq:classical-2} and \eqref{eq:classical-5} are equations~(2.1.9) and~(2.1.10), respectively, of \cite[Section~2.1]{BB98}.
The second identity in \eqref{eq:classical-5} follows from \cite[Theorem~2.1]{BB98} by analytic continuation with the branch specified above.
The transformation \eqref{eq:classical-3} is the functional equation of Jacobi's theta constants; see \cite[Chapter I, Section 7]{Mum07a}.
The value \(\vartheta_{00}(i)\) in \eqref{eq:classical-4} is given in \cite{CM23}.
By Proposition~\ref{prop:splitting}, we obtain the following corollary.

\begin{corollary}
	\label{cor:four-thetas}
	We have
	\begin{align*}
		\Theta_1 = \Theta_2
		 & = \frac{\sqrt2}{4}\,\vartheta_{00}(i)^2
		\left[ \vartheta_{00}\!\left(\tfrac{\tau^\ast}{2}\right)\vartheta_{00}\!\left(-\tfrac{1}{2\tau^\ast}\right)
			+ \vartheta_{01}\!\left(\tfrac{\tau^\ast}{2}\right)\vartheta_{01}\!\left(-\tfrac{1}{2\tau^\ast}\right) \right],
		\\
		\Theta_3 = -\Theta_4
		 & = \frac{\sqrt2}{4}\,\vartheta_{00}(i)^2
		\left[ \vartheta_{01}\!\left(\tfrac{\tau^\ast}{2}\right)\vartheta_{00}\!\left(-\tfrac{1}{2\tau^\ast}\right)
			- \vartheta_{00}\!\left(\tfrac{\tau^\ast}{2}\right)\vartheta_{01}\!\left(-\tfrac{1}{2\tau^\ast}\right) \right].
	\end{align*}
\end{corollary}
\begin{proof}
	Apply \eqref{eq:general-split} to the four characteristics of \eqref{eq:fJstar-blocked}.
	In the theta constants on the right-hand side of \eqref{eq:general-split}, we reduce the components of \(\alpha\) and \(\beta\) modulo \(2\) using \(\vartheta_{\alpha_j+2,\beta_j} = \vartheta_{\alpha_j,\beta_j}\) and \(\vartheta_{\alpha_j,\beta_j+2} = \e(\alpha_j/2)\,\vartheta_{\alpha_j,\beta_j}\), which follow from \eqref{eq:theta-series} by a shift of the summation index.
	We keep the original components of \(\alpha\) in the exponential factor and retain the factors produced by the second identity.
	For instance, \((a,b) = \bigl((1,1,0,0),\,0\bigr)\) gives \((\alpha_1,\alpha_2,\alpha_3,\alpha_4) = (2,2,2,0)\) and \((\beta_1,\beta_2,\beta_3,\beta_4) = (0,0,0,0)\), and we obtain
	\[
		\Theta_1
		=\frac14
		\left[
			\vartheta_{00}\!\left(\tfrac{\tau^\ast}{2}\right)\vartheta_{00}\!\left(-\tfrac{1}{2\tau^\ast}\right)
			+\vartheta_{01}\!\left(\tfrac{\tau^\ast}{2}\right)\vartheta_{01}\!\left(-\tfrac{1}{2\tau^\ast}\right)
			\right]
		\left[
			\vartheta_{00}\!\left(\tfrac{i}{2}\right)^2-\vartheta_{01}\!\left(\tfrac{i}{2}\right)^2
			\right].
	\]
	Applying \eqref{eq:classical-3} at \(\tau = 2i\), we have \(\vartheta_{00}(\tfrac i2) = \sqrt2\,\vartheta_{00}(2i)\) and \(\vartheta_{01}(\tfrac i2) = \sqrt2\,\vartheta_{10}(2i)\).
	Using these identities with \eqref{eq:classical-2} at \(\tau = i\) and \eqref{eq:classical-4}, we find that the second factor equals \(2\vartheta_{01}(i)^2 = \sqrt2\,\vartheta_{00}(i)^2\).
	The following table lists the characteristic \((a, b)\) and the vectors \((\alpha_1,\ldots,\alpha_4)\) and \((\beta_1,\ldots,\beta_4)\) for each of \(\Theta_1,\ldots,\Theta_4\).
	\begin{center}
		\begin{tabular}{c|c|c|c}
			             & \((a, b)\)                             & \((\alpha_1,\alpha_2,\alpha_3,\alpha_4)\)
			             & \((\beta_1,\beta_2,\beta_3,\beta_4)\)                                              \\
			\hline
			\(\Theta_1\) & \(\bigl((1,1,0,0),\, 0\bigr)\)
			             & \((2,2,2,0)\)                          & \((0,0,0,0)\)                             \\
			\(\Theta_2\) & \(\bigl(0,\, (1,1,0,0)\bigr)\)
			             & \((0,0,0,0)\)                          & \((1,1,-1,0)\)                            \\
			\(\Theta_3\) & \(\bigl((0,0,1,1),\, (1,1,1,1)\bigr)\)
			             & \((2,0,2,0)\)                          & \((1,0,0,0)\)                             \\
			\(\Theta_4\) & \(\bigl((1,1,1,1),\, (0,0,1,1)\bigr)\)
			             & \((4,2,4,0)\)                          & \((0,-1,1,0)\)
		\end{tabular}
	\end{center}
	For \(\Theta_2\), the first factor is the same as for \(\Theta_1\).
	The second factor is \(2\vartheta_{00}(\tfrac i2)\vartheta_{01}(\tfrac i2) = 4\vartheta_{00}(2i)\vartheta_{10}(2i) = 2\vartheta_{10}(i)^2 = \sqrt2\,\vartheta_{00}(i)^2\).
	Hence we have \(\Theta_2 = \Theta_1\).
	For \(\Theta_3\), the exponential factor in \eqref{eq:general-split} equals \((-1)^{\nu_1}(-1)^{\nu_2}\).
	The sum over \(\nu_1\) gives the difference \(\vartheta_{01}(\tfrac{\tau^\ast}{2})\vartheta_{00}(-\tfrac{1}{2\tau^\ast}) - \vartheta_{00}(\tfrac{\tau^\ast}{2})\vartheta_{01}(-\tfrac{1}{2\tau^\ast})\) in the stated formula for \(\Theta_3\).
	The sum over \(\nu_2\) gives \(\vartheta_{00}(\tfrac i2)^2 - \vartheta_{01}(\tfrac i2)^2 = \sqrt2\,\vartheta_{00}(i)^2\), and the two sums together yield the stated formula for \(\Theta_3\).
	For \(\Theta_4\), the exponential factor equals \((-1)^{\nu_1}\).
	Note that \(\alpha_3 + \alpha_4 = 4\).
	Using \(\vartheta_{0,2}(\tfrac i2) = \vartheta_{00}(\tfrac i2)\), we have
	\begin{align*}
		\Theta_4
		 & =\frac14
		\left[
			\vartheta_{00}\!\left(\tfrac{\tau^\ast}{2}\right)\vartheta_{01}\!\left(-\tfrac{1}{2\tau^\ast}\right)
			-\vartheta_{01}\!\left(\tfrac{\tau^\ast}{2}\right)\vartheta_{00}\!\left(-\tfrac{1}{2\tau^\ast}\right)
			\right]
		\cdot 2\vartheta_{00}\!\left(\tfrac i2\right)\vartheta_{01}\!\left(\tfrac i2\right)=-\Theta_3. \qedhere
	\end{align*}
\end{proof}

The following is the key lemma of this section.

\begin{lemma}
	\label{lem:duplication}
	We have
	\begin{equation*}
		\Theta_1\,\Theta_3
		= \frac{-i\tau^{\ast}}{4}\,\vartheta_{00}(i)^4
		\left( \vartheta_{01}(\tau^\ast)^4 - \vartheta_{10}(\tau^\ast)^4 \right).
	\end{equation*}
\end{lemma}
\begin{proof}
	By \eqref{eq:classical-3} at \(\tau=2\tau^\ast\), where \(\sqrt{-2i\tau^\ast}\) is positive, we have
	\[
		\vartheta_{00}\!\left(-\tfrac{1}{2\tau^\ast}\right)
		=\sqrt{-2i\tau^\ast}\,\vartheta_{00}(2\tau^\ast),
		\qquad
		\vartheta_{01}\!\left(-\tfrac{1}{2\tau^\ast}\right)
		=\sqrt{-2i\tau^\ast}\,\vartheta_{10}(2\tau^\ast).
	\]
	Using \eqref{eq:classical-1} at \(\tau^\ast/2\) and \eqref{eq:classical-2} at \(\tau^\ast\), we obtain
	\begin{align*}
		\vartheta_{00}\!\left(\tfrac{\tau^\ast}{2}\right)\vartheta_{00}(2\tau^\ast)
		+\vartheta_{01}\!\left(\tfrac{\tau^\ast}{2}\right)\vartheta_{10}(2\tau^\ast)
		 & =\vartheta_{01}(\tau^\ast)^2+\vartheta_{10}(\tau^\ast)^2,
		\\
		\vartheta_{01}\!\left(\tfrac{\tau^\ast}{2}\right)\vartheta_{00}(2\tau^\ast)
		-\vartheta_{00}\!\left(\tfrac{\tau^\ast}{2}\right)\vartheta_{10}(2\tau^\ast)
		 & =\vartheta_{01}(\tau^\ast)^2-\vartheta_{10}(\tau^\ast)^2.
	\end{align*}
	Substituting these four identities into the formulas of Corollary~\ref{cor:four-thetas}, we obtain
	\begin{align*}
		\Theta_1
		 & =\frac12\,\vartheta_{00}(i)^2\sqrt{-i\tau^\ast}
		\left(\vartheta_{01}(\tau^\ast)^2+\vartheta_{10}(\tau^\ast)^2\right), \\
		\Theta_3
		 & =\frac12\,\vartheta_{00}(i)^2\sqrt{-i\tau^\ast}
		\left(\vartheta_{01}(\tau^\ast)^2-\vartheta_{10}(\tau^\ast)^2\right).
	\end{align*}
	Multiplying these two equalities, we obtain the stated formula.
\end{proof}

\section{Evaluation of the Constant}
\label{sec:evaluation}
In this section, we identify the modulus \(\tau^\ast\) with a value of a Schwarz map of the Gauss hypergeometric differential equation, and complete the evaluation of the constant \(\kappa\).

A Schwarz map is a ratio of two linearly independent solutions of the Gauss hypergeometric differential equation with parameters \(\tfrac12,\tfrac12;1\), and is determined only up to a M\"obius transformation.
We use the ratio exhibited in \eqref{eq:taustar-schwarz}.
We begin with the following lemma.

\begin{lemma}
	\label{lem:hypergeometric}
	We have
	\begin{equation}
		\label{eq:taustar-schwarz}
		\tau^{\ast} = i\,\frac{\mathcal{F}(1-m)}{\mathcal{F}(m)},
	\end{equation}
	and consequently
	\begin{gather}
		\label{eq:H1}
		\sqrt{c}\;\vartheta_{00}(\tau^\ast)^4
		= \vartheta_{01}(\tau^\ast)^4 - \vartheta_{10}(\tau^\ast)^4,
		\\
		\label{eq:H2}
		\mathcal{K}_2 = \frac{\sqrt2\,\pi\,\vartheta_{00}(\tau^\ast)^2}{\sqrt{c}},
		\\
		\label{eq:H3}
		\Transpose{\eta_c}U\eta_c
		= u^2 - v^2 = -8i\tau^\ast \mathcal{K}_2^2
		= -\frac{16\pi^2 i \tau^\ast\,\vartheta_{00}(\tau^\ast)^4}{c}.
	\end{gather}
\end{lemma}
\begin{proof}
	By \eqref{eq:eta-c} and Lemma~\ref{lem:eta-c-positivity}, we have
	\[
		\tau^\ast
		=i\left(\frac{\sqrt2\mathcal{K}_1}{\mathcal{K}_2}-1\right)
		=i\,\frac{\mathcal{F}(1-m)}{\mathcal{F}(m)},
		\qquad
		u^2-v^2=-8i\tau^\ast\mathcal{K}_2^2>0.
	\]
	This proves \eqref{eq:taustar-schwarz}; the positivity of \(u^2-v^2\) is established in Lemma~\ref{lem:eta-c-positivity}.
	By \cite[Theorem~2.3(b)]{BB98}, with \(k=\sqrt m\), we have \(\lambda\bigl(i\mathcal{F}(1-m)/\mathcal{F}(m)\bigr) = m\).
	By \eqref{eq:taustar-schwarz}, we have \(\tau^\ast = i\mathcal{F}(1-m)/\mathcal{F}(m)\), and hence \(\lambda(\tau^\ast) = m\).
	Since \(\tau^\ast\in i\R_{>0}\) and \(\lambda\) restricts to a bijection from \(i\R_{>0}\) onto \((0,1)\), the modulus \(\tau^\ast\) is the unique point of \(i\R_{>0}\) at which \(\lambda\) takes the value \(m\).
	The second identity in \eqref{eq:classical-5} then gives \(\vartheta_{00}(\tau^\ast)^2 = \mathcal{F}(\lambda(\tau^\ast)) = \mathcal{F}(m)\).
	Equation~\eqref{eq:H1} follows from \(1 - 2\lambda(\tau^\ast) = \bigl(\vartheta_{01}(\tau^\ast)^4 - \vartheta_{10}(\tau^\ast)^4\bigr)/\vartheta_{00}(\tau^\ast)^4 = \sqrt c\), where we use the first identity in \eqref{eq:classical-5}.
	Equation~\eqref{eq:H2} follows from \(\mathcal{K}_2 = \sqrt2\pi c^{-1/2} \mathcal{F}(m) = \sqrt2\pi c^{-1/2}\vartheta_{00}(\tau^\ast)^2\).
	Finally, we substitute \eqref{eq:H2} into the identity \(u^2 - v^2 = -8i\tau^\ast \mathcal{K}_2^2\) established at the beginning of the proof, and obtain the last expression in \eqref{eq:H3}.
\end{proof}

\begin{proof}[Proof of the \ref{thm:main}]
	By \eqref{eq:fJstar-blocked}, Corollary~\ref{cor:four-thetas}, Lemma~\ref{lem:duplication}, and \eqref{eq:H1} of Lemma~\ref{lem:hypergeometric}, we have
	\begin{align*}
		f_{J^\ast}(\eta_c)
		 & = \frac{\vartheta_{00}(i)^8}{4}\,(2\Theta_1^2)(2\Theta_3^2)
		= \vartheta_{00}(i)^8\,(\Theta_1\Theta_3)^2                    \\
		 & = \vartheta_{00}(i)^8
		\left( \frac{-i\tau^\ast}{4}\,\vartheta_{00}(i)^4\,
		\sqrt{c}\;\vartheta_{00}(\tau^\ast)^4 \right)^2
		= -\frac{\tau^{\ast 2}}{16}\,\vartheta_{00}(i)^{16}\,
		c\,\vartheta_{00}(\tau^\ast)^8.
	\end{align*}
	On the other hand, by \eqref{eq:H3} of Lemma~\ref{lem:hypergeometric}, we have
	\[
		(\Transpose{\eta_c}U\eta_c)^2
		= -\frac{256\pi^4 \tau^{\ast2}\,\vartheta_{00}(\tau^\ast)^8}{c^2}.
	\]
	Substituting these two expressions into \eqref{eq:kappa-anchor}, we obtain
	\[
		\kappa
		= \frac{-\frac{\tau^{\ast2}}{16}\vartheta_{00}(i)^{16} c\,\vartheta_{00}(\tau^\ast)^8}
		{\left( -\frac{256\pi^4\tau^{\ast2}\vartheta_{00}(\tau^\ast)^8}{c^2} \right)(-c^3)}
		= -\frac{\vartheta_{00}(i)^{16}}{4096\,\pi^4}
		= -\frac{\vartheta_{00}(i)^{16}}{(8\pi)^4}
		= -\frac{1}{2^{12}\,\varGamma(3/4)^{16}},
	\]
	where the last equality uses \(\vartheta_{00}(i) = \pi^{1/4}/\varGamma(3/4)\).
	This computation determines the value \(\kappa\) of the function \(G\), and the \ref{thm:main} follows for every \((2,2,2,2)\)-partition \(J\).
\end{proof}

Combining the \ref{thm:main} with the continuous extensions of \(f_{J^\ast}\) and of the theta constants to the degenerate configuration, we obtain the boundary value
\[
	f_{J^\ast}(\eta_c)
	=\frac{\vartheta_{00}(i)^{16}}{(8\pi)^4}\,c^3(\Transpose{\eta_c}U\eta_c)^2
	=\frac{c^3}{2^{12}\varGamma(3/4)^{16}}\,(\Transpose{\eta_c}U\eta_c)^2.
\]

\begin{remark}
	\label{rem:limit-curve}
	The four branch points \(0,c,1,\infty\) of the limit curve \(C_c\colon w^4 = z^3(z-c)^3(z-1)\) carry the weights \(\tfrac34,\tfrac34,\tfrac14,\tfrac14\) of \(\frac{dz}{w}\), respectively.
	The sum of these weights is \(2\).
	The curves \(C_c\) with \(c\in(0,1)\) thus form one of the one-dimensional families of \cite{DM86}.
	The modulus \(\tau^\ast\) satisfies \(\lambda(\tau^\ast) = (1-\sqrt c)/2\).
	The involution \(\rho^2\colon C_c \ni (z,w) \mapsto (z,-w) \in C_c\) lifts to the normalization, and the substitution \(y = w^2/(z(z-c))\) identifies the quotient of the normalization by \(\rho^2\) with the Legendre elliptic curve \(y^2 = z(z-c)(z-1)\).
	The same degeneration performed only at \(x_4 = x_5 = x_6 = c\) leads to the three-dimensional complex ball attached to six points.
	The Thomae-type formulas for six points are the counterparts of the \ref{thm:main}, with the constant \(1/\bigl( (16\pi)^2\varGamma(3/4)^8 \bigr)\); see \cite{MN26}.
\end{remark}

\section*{Acknowledgments}
The author is deeply grateful to Professor {\sc Matsumoto} Keiji, under whose supervision this work was carried out, for guidance and constant encouragement.
The author also thanks Professor {\sc Terasoma} Tomohide for helpful advice and discussions on this work.

During the preparation of this paper, the author used Claude, an AI model developed by Anthropic, to copy-edit the English.
The author reviewed all the resulting text and takes full responsibility for the content of this paper.

\printbibliography

\end{document}